\documentclass[12pt]{article}
\usepackage[affil-it]{authblk}
\usepackage[english]{babel}
\usepackage{blindtext}

\usepackage[utf8]{inputenc}
\usepackage{amsmath,amssymb,amsthm,amsfonts}
\usepackage{bm}
\usepackage[hidelinks,colorlinks = true]{hyperref}
\usepackage{enumitem}
\usepackage{blkarray}
\usepackage{tikz-cd}
\usepackage{tcolorbox}
\newtcolorbox{mybox}{colback=red!5!white,colframe=red!75!black, sharp corners = all}
\usepackage{bclogo} %\bcpanchant
\def\ucsign{\bcpanchant}
\setenumerate{labelindent=0.25\parindent,leftmargin=*}
\usepackage[a4paper, total={16.5cm, 24cm}]{geometry}
\DeclareMathOperator{\rk}{rk}

\DeclareMathOperator{\ins}{ins}

\DeclareMathOperator{\codim}{codim}

\DeclareMathOperator{\Span}{span}
\DeclareMathOperator{\diag}{diag}
\DeclareMathOperator{\colsp}{\mathsf C}
\DeclareMathOperator{\rowsp}{\mathsf R}

\newcommand{\m}{\mathbf M}
\def\F{{\mathbb F}}
\def\K{{\mathsf K}}

\def\M{{\mathcal M}}
\def\N{{\mathbb N}}
\def\x{{\mathsf x}}
\def\SCS{{\mathbf S}}

\let\det\undefined
\DeclareMathOperator{\det}{det}

\DeclareMathOperator{\sgn}{sgn}

\DeclareMathOperator{\AS}{\mathcal{S}}

\newenvironment{psmallmatrix}
  {\left(\begin{smallmatrix}}
  {\end{smallmatrix}\right)}

\theoremstyle{plain}
\newtheorem{theorem}{Theorem}[section]
\newtheorem*{theorem*}{Theorem}
\newtheorem{lemma}[theorem]{Lemma}
\newtheorem*{lemma*}{Lemma}

\newtheorem*{proposition*}{Proposition}
\newtheorem{corollary}[theorem]{Corollary}
\newtheorem*{corollary*}{Corollary}

\theoremstyle{definition}
\newtheorem{definition}[theorem]{Definition}

\newtheorem*{counterexample*}{Counterexample}
\newtheorem{remark}[theorem]{Remark}
\newtheorem{note}[theorem]{Note}
\numberwithin{equation}{section}

\begin{document}

\title{Linear varieties and matroids with applications to the Cullis' determinant}
%\author{Alexander Guterman
%\and
%Andrey Yurkov}

\author[1]{Alexander Guterman}
\author[2,*]{Andrey Yurkov}
\affil[1,2]{Department of Mathematics, Bar Ilan University, Ramat-Gan, 5290002, Israel}
\affil[*]{Corresponding author: Andrey Yurkov, andrey.yurkov@biu.ac.il}

\maketitle

\begin{abstract}
Let $V$ be a vector space of rectangular $n\times k$ matrices annihilating the Cullis' determinant. We show that $\dim(V) \le (n-1)k$, extending the Dieudonn{\'{e}}'s result on the dimension of vector spaces of square matrices annihilating the ordinary determinant.

Furthermore, for certain values of $n$ and $k$, we explicitly describe such vector spaces of maximal dimension. Namely, we establish that if $k$ is odd, $n \ge k + 2$  and $\dim(V) = (n-1)k$, then $V$ is equal to the space of all $n\times k$ matrices whose alternating row sum is equal to zero.

To obtain these results, we provide and use a notion of matroid corresponding to a given linear variety which properties may have an independent interest. In particular, we establish that if the linear variety is transformed by projections and restrictions, then the behaviour of the corresponding matroid is expressed in the terms of matroid contraction and restriction. 

In addition, we show that if $M$ is a matroid, $I^*$ its coindependent set $M|S$ and its restriction  on a set $S$, then the union of $I^*\setminus S$ with every cobase of $M|S$ is coindependent set of $M$. We employ this fact in our proofs and anticipate its utility in other studies involving matroid theory.
\end{abstract}


\section{Introduction}\label{sec:intro}

The determinant of a matrix is used in many fields of mathematics and its applications. Its study has a long and rich history, with contributions from mathematicians across the world since ancient times.

One of the subjects in the investigations of the determinant is describing the structure of spaces of matrices annihilating it,
i.e., vector spaces $V \subseteq \M_{n}(\F)$ such that $\det (X) = 0$ for all $X \in V$. Dieudonn{\'{e}} in~\cite{Dieudonne1948} obtained the sharp bound for the dimension of matrix space having this property and established that every such space of maximal possible dimension should be a left or right maximal ideal in the ring of all  matrices. 

\begin{theorem}[{\cite[Th\'{e}or\`{e}me~1]{Dieudonne1948}}]\label{thm:Dieudonne}Assume that $n \in \N$ and $\F$ is a field such that $(|\F|, n) \neq (2,2)$. Let $V \subseteq \M_n(\F)$ be a vector space, $A \in \M_{n\,k}(\mathbb F)$ and $\K = \{A\} + V$. Then the following statements hold:
\begin{enumerate}
\item If $\det(X) = 0$ for all $X \in \K$, then $\dim (V) \le n^2 - n$.
\item If  $\det(X) = 0$ for all $X \in \K$ and $\dim (V) = n^2 - n$,  then either the nullspaces of the matrices from $\K$ have a common nonzero vector, or the same is true for $\K^t$.
\end{enumerate}
\end{theorem}

This allowed Dieudonn{\'{e}} to provide a description of linear maps preserving the matrix determinant without any restrictions on ground field~\cite[Th\'{e}or\`{e}me~3]{Dieudonne1948}.

Since the ordinary determinant is defined only for the square matrices, there were made several attempts to extend this notion to the rectangular matrices. One of such extension is due to Cullis who introduced the concept of determinant (he called it \emph{determinoid}) of a rectangular matrix in his monograph~\cite{cullis1913} which can be expressed as an alternating sum of its maximal minors (see Definition~\ref{def:CullisDet}) and is denoted by $\det_{n\,k}$. Several properties known for the classical determinant were studied and shown to be valid for the Cullis' determinant in~\cite[\textsection 5, \textsection 27, \textsection 32]{cullis1913}. 

In 1966 Radi\'{c}~\cite{radic1966} independently proposed an equivalent definition of the Cullis determinant. He also provided in~ \cite{radic2005, susanj1994} its geometrical applications to the calculating areas of polygons on the plane and volumes of certain polyhedrons in space. According to that, it is also sometimes called \emph{Radi\'{c}'s determinant}~\cite{amiri2010} or \emph{Cullis-Radi\'{c} determinant}~\cite{makarewicz2014,makarewicz2020} in some papers. 

Nakagami and Yanai in~\cite{NAKAGAMI2007422} formulate the definition of the Cullis' determinant in the terms of vectors in a Grassmann algebra and provide a definition of the Cullis' determinant through several characteristic properties. For this reason they consider $n\times k$ matrices only with $n \ge k$, whereas the original definition covers rectangular matrices of any size. But according to it, the Cullis determinant of $n \times k$ matrix with $n < k$ is equal to the Cullis determinant of its transpose and consequently it is sufficient to consider only the case $n \ge k$. 

Since the Cullis' determinant is a generalization of the ordinary discriminant, it is natural to ask whether Theorem~\ref{thm:Dieudonne} admits a suitable extension. In this paper the authors make the first steps in that direction of research. In particular, we obtain a generalization of the Dieudonn{\'{e}}'s theorem for the Cullis' determinant to the case where $n \ge k + 2$ and $k$ is odd. The main result of this paper is formulated as follows.

\begin{theorem}[See~Theorem~\ref{thm:MaxDimKEvenAltSumZero} of this text]\label{thm:mainthm1}Let $n \ge k + 2$ and $V \subseteq \M_{n\,k}(\mathbb F)$ be a vector space $A \in \M_{n\,k}(\mathbb F)$ and $\K = \{A\} + V$. Then $\det_{n\,k} (X) = 0$ for all $X \in \K$ and $\codim(\K) = k$ if and only if $k$ is odd and alternating row sum of every $X \in \K$ is equal to zero.
\end{theorem}
Thus, this theorem implies that if $n \ge k + 2$ and $k$ is odd, there is a \textbf{unique} vector subspace of $\M_{n\,k}(\F)$ of maximal dimension annihilating $\det_{n\,k}$. 

The argumentation presented in our paper is similar to that used by Dieudonn{\'{e}} to prove Theorem~\ref{thm:Dieudonne}. The main difference is that the authors rely on the following considerations from the matroid theory. First, we provide a notion of matroid corresponding to a given linear variety (Definition~\ref{def:matLinVar}). Second, we prove that if the linear variety is transformed by projections and restrictions, then the behaviour of the corresponding matroid is expressed in the terms of matroid contraction and restriction (Lemma~\ref{lem:ALSLifting} and Lemma~\ref{lem:LinVarSecProj}). Third, we establish that if $M$ is a matroid, $I^*$ its coindependent set $M|S$ and its restriction  on a set $S$, then the union of $I^*\setminus S$ with every cobase of $M|S$ is coindependent set of $M$ (Lemma~\ref{lem:MatroidCobaseRestriction}).
%, which provides the proper definitions and theorems to make our claims and proofs precise and clear. This approach also provides a possibility to restate the Dieudonn{\'{e}}'s proof to meet the modern standards of rigorousness. 

Note that these observations may be useful for solving other similar problems. %, for example, for finding a description of vector spaces of matrices annihilating the matrix permanent. 
It is also worth to mention that the considerations from the matroid theory were already used to solve the problems of such kind. For example, Meshulam in~\cite{Meshulam2002} applies the Lov\'asz Matroid Parity Theorem to find the maximal  dimension of vector space $V$ consisting of skew-symmetric matrices such that every element of $V$ is a singular matrix. %using 
%It is also worth to mention that the matroid theory already showed its usefullness in solving such king of problems here that the matroid theory is already showed its prominence in solving similar problems (for exsee).
%\end{underconst}

Similarly to Theorem~\ref{thm:Dieudonne}, Theorem~\ref{thm:mainthm1} finds its applications in studying linear maps that preserve the Cullis' determinant. Using this theorem, it is possible to obtain a description of linear maps preserving the Cullis' determinant of $n\times k$ rectangular matrices for the case where $k = 3$ and $n \ge 5$, see~\cite{Guterman2025c}. %The corresponding paper~\cite{Guterman2025c} is submitted for publication. %Together with ... this comprises the full description

In addition to finding a precise description of vector spaces of matrices of size $n\times k$ annihilating the Cullis' determinant for $k$ odd we also provide the following upper bound of the dimension of such spaces for arbitrary $k$.

\begin{theorem}[Cf.~Theorem~\ref{thm:ALSDetNKZeroCodimGEK}]\label{thm:mainthm2} Let $V \subseteq \M_{n\,k}(\F)$ be a vector space, $A \in \M_{n\,k}(\mathbb F)$ and $\K = \{A\} + V$ be such that $\det_{n\,k} (X) = 0$ for every $X \in \K.$ Then $\codim (V) \ge k.$
\end{theorem}

Here $\codim (V)$ denotes a \emph{codimension} of vector space $V \subseteq \M_{n\,k}(\F)$, which is defined by $\codim(V) = \dim\left(\M_{n\,k}(\F)\right) - \dim (V) = n\cdot k - \dim (V).$

Theorem~\ref{thm:mainthm1} and Theorem~\ref{thm:mainthm2} imply in particular that if $k$ is even and $n \ge k + 2$, then the codimension of any vector space of matrices of size $n\times k$ annihilating the Cullis' determinant is strictly greater than $k$. This is the sharpest bound known by authors at present, and the question of its exact value remains open. %На данный момент это вся информация, которая имеется, и любой прогресс был бы интересен.
%This is all the information that the authors currently have regarding this case. There is not only no hypothesis about the description of such spaces of maximal dimension for even $k$, but also no strict lower bound of codimension of such spaces.

%\begin{underconst}
%ABOUT THE CASE n = k + 1.

In accordance with the text above, the two cases ($n = k$ and $n = k + 1$) remain undiscussed. The case if $n = k$ is completely covered by Theorem~\ref{thm:Dieudonne} because the Cullis' determinant of square matrix is equal to the ordinary determinant. The case $n = k + 1$, in turn, will be considered in the separate paper because it is carried out using different methods.
%, in turn, is not considered in this paper, although it is possible to reduce it to the case where $n = k$ using the methods developed in~\cite{Guterman2024}.
%\begin{underconst}The size of paper is enough to publish it as a separate unit.
%\end{underconst}
%\end{underconst}

It is worth mentioning that Theorem~\ref{thm:Dieudonne} admits other generalizations. For example, Pazzis in~\cite{DESEGUINSPAZZIS2025393} provides a description of vectors spaces of matrices of size $n \times k$ (where $n \ge k$) of maximal dimension which do not contain a matrix of rank $k$. Unfortunately, it is not possible to apply his results directly to the spaces of matrices annihilating the Cullis' determinant because there exist matrices of rank $k$ belonging to $\M_{n\,k}(\F)$ such that their Cullis' determinant is zero. For example, this is true for the matrix $E_{1\,1} + \ldots + E_{k\,k} + E_{(k+1)\, k}$. 

This paper is organized as follows. In Section~\ref{sec:notation} we provide the necessary notations; in Section~\ref{sec:prelimlinal} we provide an explanation of all necessary facts from the theory of linear varieties; Sections~\ref{sec:prelimmatroid} and~\ref{sec:prelimcul} contain the preliminary facts from the matroid theory and the theory of the Cullis' determinant, correspondingly; in Section~\ref{sec:matroidlinvar} we introduce the matroid corresponding to a given linear variety, study the behaviour of this matroid while the corresponding linear variety is transformed using different operations and establish the relationship between matroids corresponding to linear varieties and vector matroids corresponding to matrices; Section~\ref{sec:linvarsann} is devoted to investigations of linear varieties of matrices annihilating the Cullis' determinant and contains the proofs of main results of this paper; in Section~\ref{sec:furtherwork} we discuss the possible further work. 

\section{Notation and basic definitions}\label{sec:notation}

By $\F$ we denote a field without any restrictions on its characteristic and cardinality. 

We denote by $\M_{n\, k}(\F)$ the set of all $n\times k$ matrices with the entries from a certain field $\F.$ In addition, if $A$ and $B$ are finite sets, then $\M_{A\, B}(\F)$ denotes the set of all $|A|\times |B|$ matrices with the entries from $\F$ whose rows and columns are indexed by the elements of $A$ and $B$, respectively. $O_{n\, k}\in \M_{n\, k}(\F)$ denotes the matrix with all the entries equal to zero. $I_{n\, n} = I_{n} \in \M_{n\, n}(\F)$ denotes an identity matrix. Let us denote by $E_{ij} \in \M_{n\, k}(\mathbb F)$ a matrix, whose entries are all equal to zero besides the entry on the intersection of the $i$-th row and the $j$-th column, which is equal to one. By $x_{i\, j}$ we denote the element of a matrix $X$ lying on the intersection of its $i$-th row and $j$-th column. For a set $S$ and $i, j \in S$ by $\delta_{i\,j}$ we denote a \emph{Kronecker delta} of $i$ and $j$, which is equal to $1$ if $i = j$ and is equal to $0$ otherwise. If $A \in \M_{n\,k}(\F),$ then by $\rowsp (A) \subseteq \F^k$ and $\colsp (A) \subseteq \F^n$ we denote \emph{a row space of} $A$ and \emph{a column space of} $A$, correspondingly.

For $A \in \M_{n\,k_1}(\mathbb F)$ and $B \in \M_{n\,k_2}(\mathbb F)$ by $A|B \in \M_{n\,k_1+k_2}(\mathbb F)$ we denote a block matrix defined by $A|B = \begin{pmatrix} A & B\end{pmatrix}$.

For $A \in \M_{n\,k}(\mathbb F)$ by $A^t \in \M_{k\,n}(\mathbb F)$ we denote a transpose of the matrix $A$, i.e. $A^{t}_{i\,j} = A_{j\,i}$ for all $1 \le i \le k, 1 \le j \le n$.

We use the notation for submatrices following~\cite{Minc1984} and~\cite{Pierce1979}. That is, by $A[J_1|J_2]$ we denote the $|J_1| \times |J_2|$ submatrix of $A$ lying on the intersection of rows with
the indices from the set $J_1$ and the columns with the indices from the set $J_2$. By $A(J_1|J_2)$ we denote a submatrix of $A$ derived from it by striking out from it the rows with indices belonging to $J_1$ and the columns with the indices belonging to $J_2$. If the set is absent, then it means empty set, i.e. $A(J_1|)$ denotes a matrix derived from $A$ by striking out from it the rows with indices belonging to $J_1$. We may skip curly brackets, i.e. $A[1,2|3,4] = A[\{1,2\}|\{3,4\}]$. The notation with mixed brackets is also used, i.e. $A(|1]$ denotes the first column of the matrix $A$. This notation is also used for vectors as well. In this case vectors are considered as $n\times 1$ or $1 \times n$ matrices. 

We use the bold font to denote vectors and lower indices to denote their coordinates. In the case if we need the series of vectors, we use the upper indices placed in braces. For example, if $\mathbf v = \begin{pmatrix}1 \\ 0\end{pmatrix},$ then $\mathbf v^t = \begin{pmatrix}1 & 0\end{pmatrix}$ and $\mathbf v_1 = 1$. If $\mathbf u^{(1)} = \begin{pmatrix}1 \\ 0\end{pmatrix}, \mathbf u^{(2)} =  \begin{pmatrix}0 \\ 1\end{pmatrix}$, then $\mathbf u^{(1)}_1 = 1$ and $\mathbf u^{(2)}_1 = 0$.

\section{Preliminaries from the theory of linear varieties}\label{sec:prelimlinal}
In order to use the matroid theory we provide a consistent explanation of the theory of linear varieties (Definition~\ref{def:linvty}) and the relationship between their equational representation (Definition~\ref{def:LinVarEqRepr}) following~\cite{Friedberg2002-wu}, \cite{gruenberg_linear_1977} and~\cite{Dieudonne1948}.

\begin{definition}[{\cite[Observation~3 on p.~8]{Tarrida2011-ts}}]
\label{def:linvty}Let $V \subseteq W$ be  vector spaces. Then a set 
\[
\K = \{\mathbf s\} + V = \{\mathbf s + \mathbf v \colon \mathbf v \in V\}
\]
is called \emph{a linear variety}. %(see~\cite{Tarrida2011-ts} or~\cite{Dieudonne1948}) %\emph{translated subspace}(\emph{vari\'et\'e line\'aire} in~\cite{Dieudonne1948} or in).% (Observation~3 on p.~8)
 In particular, every vector subspace of $W$ is a linear variety. Linear varieties are also called~\emph{translated subspaces} (e.g.~\cite[\S 2.1, Definition of p.~15]{gruenberg_linear_1977}).

 %\cite[\S 2.1, Definition of p.~15]{gruenberg_linear_1977}}
\end{definition}

We  extensively use the properties of linear varieties in the next sections and therefore revise them in some detail. The main reason why we cannot use only the theory of vector spaces is the induction step in the proof of Lemma~\ref{lem:ALSDetNKZeroDimKThenExistVec}, where we yield one linear variety from another by equating some coordinates of its elements to nonzero elements of $\F$.

\begin{lemma}[Cf.~{\cite[\S 2.1, Lemma~1 on p.~15]{gruenberg_linear_1977}}]\label{lem:LinearVarietyShift}The following statements are equivalent:
\begin{enumerate}[label=(\arabic*), ref=(\arabic*)]
\item $\mathbf s + V = \mathbf s' + V$;
\item $\mathbf s' \in \mathbf s + V$. 
\end{enumerate}
\end{lemma}

\begin{lemma}[Cf.~{\cite[\S 2.1, Lemma~2 on p.~16]{gruenberg_linear_1977}}]\label{lem:LinearVarsEqual} If $\mathbf s + V = \mathbf s' + V'$, then $V = V'$.
\end{lemma}

\begin{definition}[Cf.~{\cite[\S 2.1, p.~16]{gruenberg_linear_1977}}]\label{def:SbelLV}Let $V \subseteq W$ be  vector spaces, $\K \subset W$ be a linear variety. The space $V$ is called \emph{a subspace belonging to} $\K$ if $\K = \mathbf s + V$ for some $\mathbf s \in W$. In~\cite{Tarrida2011-ts} the space $V$ is called \emph{a vector space associated with} $\K$.
\end{definition}

Lemma~\ref{lem:LinearVarsEqual} implies that Definition~\ref{def:SbelLV} is unambiguous.

\begin{lemma}[Cf.~{\cite[\S 2.1, Proposition~1 on p.~16]{gruenberg_linear_1977}}]\label{lem:ASIntersec} A non-empty intersection of any family of linear varieties is a linear variety.
\end{lemma}
\begin{lemma}[Cf.~{\cite[\S 2.1, proof of Proposition~1 on p.~16]{gruenberg_linear_1977}}]\label{lem:ASIntersecVecBelong}Let $\left(\K_i\;\mid\;i \in I\right)$ be a family of linear varieties having non-empty intersection, and $V_i$ be a vector space belonging to $\K_i$. Then $\bigcap\limits_{i \in I} V_i$ belongs to $\bigcap\limits_{i \in I} \K_i$.
\end{lemma}

\begin{definition}[Cf.~{\cite[\S 2.1, Definition on p.~16]{gruenberg_linear_1977}}]Let $\K$ be a linear variety and $\K =\mathbf s + V$. Then $\dim(V)$ is called a \emph{dimension} of $\K$ and denoted by $\dim(\K)$.
\end{definition}

\begin{definition}Let $\K$ be a linear variety and $\K =\mathbf s + V$. Then $n - \dim(V)$ is called a \emph{codimension} of $\K$ and denoted by $\codim(\K)$.
\end{definition}

\begin{lemma}[Cf.~{\cite[\S 2.2, Theorem~1 on p.~19]{gruenberg_linear_1977}}]\label{lem:ASSubAS}Let $\K = \mathbf s + V$ and $\K' = \mathbf s + V'$ are linear varieties and $\K \subseteq \K'$. Then
\begin{enumerate}[label=(\alph*), ref=(\alph*)]
\item\label{lem:ASSubAS:part2} $\codim(\K) \ge \codim(\K')$;
\item\label{lem:ASSubAS:part3} if $\codim(\K) = \codim(\K')$, then $\K = \K'$.
\end{enumerate}
\end{lemma}

\begin{definition}For $A \in \M_{m\,n}(\F),\mathbf b \in \F^{n}$ we denote by $\AS(A,\mathbf b)$ the solution set of a system of linear equations $A\mathbf x = \mathbf b$.
\end{definition}

\begin{lemma}[Cf.~{\cite[Sec.~3.3, Theorem~3.8]{Friedberg2002-wu}}]\label{lem:HomEqVecSpaceDim}Let $A \in \M_{m\,n}(\F)$ and $A\mathbf x = 0$ be a homogeneous system of $m$ linear equations in $n$ unknowns over a field $\F$. Then $\AS(A,\mathbf 0)$ is a vector subspace of $\F^n$ of dimension $n - \rk(A)$.
\end{lemma}

\begin{lemma}[See~{\cite[Sec.~3.3, Theorem~3.9]{Friedberg2002-wu}}]\label{lem:NonHIsSumH}Let $A \in \M_{m\,n}(\F), \mathbf b \in \F^{n}$. Then for any solution $\mathbf s$ to $A\mathbf x = \mathbf b$
\[
\AS(A,\mathbf b) = \{\mathbf s\} + \AS(A,\mathbf 0) = \{\mathbf s + \mathbf k \colon \mathbf k \in \AS(A,\mathbf 0)\}.
\]
\end{lemma}

\begin{corollary}\label{cor:SolSetLinVty}Let $A \in \M_{m\,n}(\F), \mathbf b \in \F^{n}$. If a system of linear equations $A\mathbf x = \mathbf b$ is consistent, then its solution set is a linear variety and $\AS(A,\mathbf 0)$ is a vector space belonging to it.
\end{corollary}

\begin{lemma}[See~{\cite[Sec.~3.3, Theorem~3.11]{Friedberg2002-wu}}]\label{lem:KronCap}Let $A \in \M_{m\,n}(\F),\mathbf b \in \F^{n}$ and  $A\mathbf x = \mathbf b$ be a system of linear equations. Then the system is consistent if and only if $\rk(A) = \rk(A|\mathbf b)$.
\end{lemma}

\begin{definition}[See~{\cite[Sec~3.4, Definition on p. 182]{Friedberg2002-wu}}]\label{lem:InvMultEquiv}
Two systems of linear equations are called \emph{equivalent} if they have the same solution set.
\end{definition}

\begin{lemma}[See~{\cite[Sec.~3.4, Theorem~3.13]{Friedberg2002-wu}}]Let $A\mathbf x = \mathbf b$ be a system of $m$ linear equations in $n$ unknowns, and let $C$ be an invertible $m\times m$ matrix. Then the system $(CA)\mathbf x = C\mathbf b$ is equivalent to $A\mathbf x = \mathbf b$.
\end{lemma}

\begin{definition}\label{def:LinVarEqRepr}Let $A \in \M_{m\,n}(\F), \mathbf b \in \F^{n}$ and $\K \subseteq \F^n$ be a linear variety. If  $\K = \AS(A,\mathbf b)$, then the pair $(A,\mathbf b)$ is called an \emph{equational representation} for $\K$. 
\end{definition}

\begin{lemma}[Cf.~{\cite[\S 3.3, Proposition~2]{gruenberg_linear_1977}}]\label{lem:ForLinVarExistsEqRep}
Every linear variety has an equational representation. 
\end{lemma}

\begin{lemma}\label{lem:MatRankEqToCodim}Let $(A,\mathbf b)$ be an equational representation of a linear variety $\K \subseteq \F^n$. Then $\rk(A) = \codim(\K)$.
\end{lemma}
\begin{proof}Let $\mathbf s \in \K$. Lemma~\ref{lem:NonHIsSumH} implies that $\K = \mathbf s + \AS(A,\mathbf 0)$. Then
\[
\rk (A) = n - \dim(\AS(A,\mathbf 0)) = \codim (\K)
\]
by Lemma~\ref{lem:HomEqVecSpaceDim} because $\AS(A,\mathbf 0)$ is a vector space which also follows from Lemma~\ref{lem:HomEqVecSpaceDim}.
\end{proof}

\begin{lemma}\label{lem:ChooseFullRank}Let $\K$ be a linear variety,  $A \in \M_{m\, n}(\F), \mathbf b \in \F^{m}$ be such that 
\begin{equation}\label{lem:ChooseFullRank:eqq1}
\K = \AS(A,\mathbf b)\subsetneq \F^{n}.
\end{equation} 
Then there exists $R \subseteq [m]$ such that $\K = \AS\Bigl(A[R|),\mathbf b[R|)\Bigr)$ and $A[R|)$ has full rank.
\end{lemma}
%
%\begin{proof}Let $\K_H$ be the solution set of the corresponding homogeneous system $A\mathbf x = \mathbf 0$. Lemma~\ref{lem:NonHIsSumH} implies that $\K_H \neq \F^n$. According to Lemma~\ref{lem:HomEqVecSpaceDim}, the linear variety $\K_H$ is a vector subspace of $\F^n$ and $\rk(A) = r = n - \dim(\K_H)$. Since $\K_H \neq \F^n$, then $\dim(\K_H) < n$ and consequently $r > 0$.
%
%Let $R = \{i_1,\ldots, i_r\}$ be any set of indices of linearly independent rows in $A$. Since $\rk(A) = r$, every equation of system $Ax = \mathbf b$ is a linear combination of equations with indices belonging to $R$. Therefore, $\K = \AS(A[R|),\mathbf b[R|))$ for this $R$.
%\end{proof}

\begin{proof} Lemma~\ref{lem:NonHIsSumH} implies that 
\begin{equation}\label{lem:ChooseFullRank:eq1}
\AS(A, \mathbf 0) \subsetneq \F^n.
\end{equation} 
According to Lemma~\ref{lem:HomEqVecSpaceDim}, $V = \AS(A, \mathbf 0)$ is a vector subspace of $\F^n$. Hence, $\dim(V) < n$ by~\eqref{lem:ChooseFullRank:eq1}. Therefore, relying on Lemma~\ref{lem:HomEqVecSpaceDim} again we conclude that $\rk(A) = r = n - \dim(V) > 0$. 

Let $R = \{i_1,\ldots, i_r\}$ be any set of indices of linearly independent rows in $A$. The equality~\eqref{lem:ChooseFullRank:eqq1} implies that $\AS(A, \mathbf b) \neq 0$. Hence, the system of linear equations $Ax = \mathbf b$ is consistent. Therefore, since $\rk(A) = r$, from the properties of rank of the matrix $A|\mathbf b$ we obtain that every equation of system $Ax = \mathbf b$ is a linear combination of equations with indices belonging to $R$. Thus, $\K = \AS(A[R|),\mathbf b[R|)).$
\end{proof}

\begin{lemma}\label{lem:EqRepEqRank}Suppose that $\varnothing \neq \AS(A,\mathbf b) = \AS(A',\mathbf b') \subseteq \F^n$. Then
\begin{enumerate}[label=(\alph*), ref=(\alph*)]
\item\label{lem:EqRepEqRank:part1} $\AS(A,\mathbf 0) = \AS(A',\mathbf 0)$;
\item\label{lem:EqRepEqRank:part2} $\rk(A) = \rk(A')$;
\item\label{lem:EqRepEqRank:part3} if $A$ and $A'$ have full rank, then there exists an invertible matrix $C \in \M_{\rk(A)\, \rk(A)}(\F)$ such that $CA = A'$ and $C\mathbf b = \mathbf b'$.
\end{enumerate}
\end{lemma}
\begin{proof}\textbf{\ref{lem:EqRepEqRank:part1}}\quad Let $\mathbf s \in \AS(A,\mathbf b)$. Since $\AS(A,\mathbf b) = \AS(A',\mathbf b')$, then $\mathbf s \in \AS(A',\mathbf b')$. It follows from Lemma~\ref{lem:NonHIsSumH} that
\[
\{\mathbf s\} + \AS(A,\mathbf 0) = \AS(A,\mathbf b)
\;\;\;\;\mbox{and}\;\;\;\;
\{\mathbf s\} + \AS(A',\mathbf 0) = \AS(A',\mathbf b').
\]
Then the assumption of the lemma implies that
\[
\{\mathbf s\} + \AS(A,\mathbf 0) = \{\mathbf s\} + \AS(A',\mathbf 0).
\]
Therefore, since $\AS(A,\mathbf 0)$ and $\AS(A',\mathbf 0)$ are vector spaces by Lemma~\ref{lem:HomEqVecSpaceDim},
\begin{equation}\label{lem:EqRepEqRank:eq1}
\AS(A,\mathbf 0) = \AS(A',\mathbf 0)
\end{equation}
 by Lemma~\ref{lem:LinearVarsEqual}.

\paragraph{\ref{lem:EqRepEqRank:part2}} Lemma~\ref{lem:HomEqVecSpaceDim} implies that $\AS(A,\mathbf 0)$ and $\AS(A',\mathbf 0)$ are both vector spaces. Then the equality~\eqref{lem:EqRepEqRank:eq1} implies that $\dim(\AS(A,\mathbf 0)) = \dim(\AS(A',\mathbf 0)).$ Therefore,
\[
\rk(A) = n - \dim(\AS(A,\mathbf 0)) = n - \dim(\AS(A',\mathbf 0)) = \rk(A')
\]
by Lemma~\ref{lem:HomEqVecSpaceDim}.

\paragraph{\ref{lem:EqRepEqRank:part3}}Let $A'' = \begin{psmallmatrix}A\\ A'\end{psmallmatrix}$ and $\mathbf b'' = \begin{psmallmatrix}\mathbf b\\ \mathbf b'\end{psmallmatrix}$. Since $\AS(A,\mathbf b)= \AS(A',\mathbf b')$, 
\[
\AS(A'',\mathbf b'') = \AS(A,\mathbf b) \cap \AS(A',\mathbf b')= \AS(A,\mathbf b).
\]
Hence, $\rk(A'') = \rk(A)$ by the part~\ref{lem:EqRepEqRank:part2} of the lemma. Therefore, every row of $A''$ corresponding to $A'$ is a linear combination of rows corresponding to $A$, which means that there exists a matrix $C \in \M_{\rk(A)\,\rk(A)}(\F)$ such that
\begin{equation}\label{lem:EqRepEqRank:eq111}
CA = A'.
\end{equation}
Since $\rk(A) = \rk(A')$, this matrix is invertible.

Let $\mathbf s$ be any solution of the system of linear equations $Ax = \mathbf b$. It exists because we suppose that $\AS(A,\mathbf b)\neq \varnothing$. Thus, $A\mathbf s = \mathbf b$. This equality together with the equality~\eqref{lem:EqRepEqRank:eq111} imply that
\[
C\mathbf b = CA\mathbf s = A'\mathbf s = \mathbf b'
\]
because $\AS(A,\mathbf b) = \AS(A',\mathbf b')$ and consequently $\mathbf s \in \AS(A',\mathbf b')$ as well.
\end{proof}

\begin{corollary}\label{cor:ArbitraryFullRank}Let $A \in \M_{m\, n}(\F), \mathbf b \in \F^{n}$ and  $\K = \AS(A,\mathbf b)\subsetneq \F^{n}$ be a linear variety. If $R \subseteq [m]$ is a nonempty set such that $A[R|)$ has full rank, then $\K = \AS(A[R|),\mathbf b[R|))$.
\end{corollary}

\begin{lemma}\label{lem:LinImageAS}Let $\K \subseteq \F^n$ be a linear variety, $V \subseteq \F^n$ be a vector space belonging to $\K$ and $T \colon \F^n \to \F^m$ be a linear map. Then $T(\K)$ is a linear variety and $T(V)$ is a vector space belonging to $T(\K)$.
\end{lemma}
\begin{proof}Indeed, if $\K = \mathbf s + V$ for some $\mathbf s \in \F^n$ and a vector subspace $V \subseteq \F^n$, then $T(\K) = T(\mathbf s) + T(V)$, where $T(\mathbf s) \in \F^n$ and $T(V) \subseteq \F^n$ is a vector subspace of $\F^m$.
\end{proof}

\begin{lemma}\label{lem:InvertibleLDAS}Let $\K$ be a linear variety, $A \in \M_{m\, n}(\F)$, $\mathbf b \in \F^m$ be such that $\K = \AS(A,\mathbf b)$,  $C \in \M_{n\,n}(\F)$ be an invertible matrix and $\K_C = \{C\mathbf x \mid \mathbf x \in \K\}.$ Then
\begin{enumerate}[label=(\alph*), ref=(\alph*)]
\item\label{lem:InvertibleLDAS:part1} $\K_C = \AS(AC^{-1},\mathbf b)$;
\item\label{lem:InvertibleLDAS:part2} $\codim(\K_C) = \codim(\K)$.
\end{enumerate}
\end{lemma}
\begin{proof}\textbf{\ref{lem:InvertibleLDAS:part1}}\quad Indeed, the conditions
$
AC^{-1}(C\mathbf x) = \mathbf b\;\;\mbox{and}\;\; A\mathbf x = \mathbf b
$
are equivalent for all $\mathbf x \in \F^{n}$.
\paragraph{\ref{lem:InvertibleLDAS:part2}}Assume that $\K = \mathbf s + V$ for some $\mathbf s \in \F^n$ and vector subspace $V \subseteq \F^n$. Lemma~\ref{lem:LinImageAS} implies that $V_C = \{C\mathbf x \mid \mathbf x \in V\}$ is a vector space belonging to $\K_C$. Since $C$ is invertible, $\dim(V_C) = \dim(V)$, which implies the required equality.
\end{proof}

We also need the definition for a vector space over $\F$ whose coordinates are indexed by an arbitrary finite set.

\begin{definition}Let $E$ be a finite set. Then by $\F^E$ we denote the set of all functions from $E$ to $\F$ endowed with the standard structure of vector space over $\F$; that is, $$(f + \lambda g)(x) = f(x) + \lambda g(x)$$ for all $f, g \in \F^E, \lambda \in \F, x \in E$. We call this space a \emph{coordinate space indexed by} $E$.

For every $i \in E$ let us also denote by $\mathbf e_i$ the function $E \to \F$ which is equal to 1 if its argument is equal to $s$, otherwise it is equal to 0. Since $E$ is finite, the set $\{\mathbf e_i \mid i \in E\}$ forms a basis of $\F^E$.

For every $\mathbf v \in \F^{E}$ and $i \in E$ we denote by $\mathbf v_i$ the value of the function identified with $\mathbf v$ at the point $i$.
\end{definition}

It is clear that $\F^n$ could also be considered as $\F^{[n]}$ and for every finite set $S$ the space $\F^S$ could be considered as $\F^{|S|}$. Therefore, all the theory developed above could be applied to $\F^S$ with $S$ finite. We also identify $\M_{n\,k}(\F)$ with $\F^{[n]\times [k]}$ by the correspondence $E_{i\,j} \leftrightarrow \mathbf e_{(i,j)}$.

\begin{definition}\label{def:projEEP}If $E' \subset E$ are finite sets, then by $\pi_{E}^{E'} \colon \F^E \to \F^{E'}$ we denote \emph{a standard projection} from $\F^{E}$ on $\F^{E'}$. That is, if $\mathbf x \in \F^E$, then $\pi_{E}^{E'}(\mathbf x)$ is defined by
\[
\pi_{E}^{E'}(\mathbf x)_i = \mathbf x_i\;\;\mbox{for all}\;\; i \in E'.
\]
\end{definition}

\begin{definition}\label{def:ProjectionMap}
Let $E_1$ and $E_2$ be two finite sets and $f\colon E_2 \to E_1$ be an injective function. Denote by $\pi_f \colon \F^{E_1} \to \F^{E_2}$ a linear map defined by
\[
\pi_f(\mathbf v)_{a} = \mathbf v_{f(a)}\;\;\mbox{for all}\;\; \mathbf v \in \F^{E_1}\;\;\mbox{and}\;\;a \in E_2.
\]
We call this map \emph{$f$-projection}. Since $f$ is injective, $\pi_f$ is surjective.
\end{definition}

\begin{definition}If $E$ is a finite set, $i \in E$ its element, then by $\x_i^E \colon \F^E \to \F$ we denote $i$-coordinate function. That is,
\[
\x_i^E(\mathbf x) = \mathbf x_i\;\;\mbox{for all}\;\; \mathbf x \in \F^E.
\]
Thus, the set $\{\x_i \mid i \in E\}$ is a basis of the dual space ${\F^E}^*$.
\end{definition}

\section{Preliminaries from the matroid theory}\label{sec:prelimmatroid}

Let us recall the necessary facts from the matroid theory following~\cite{Oxley2011}.

\begin{definition}[{\cite[Definition on p.7]{Oxley2011}}]\label{def:matroid}A \emph{matroid} $\m$ is an ordered pair $(E, \mathcal I)$ consisting of a finite set $E$ and a collection $\mathcal I$ of subsets of $E$ having the following three properties:
\begin{enumerate}[label=(I\arabic*), ref=(I\arabic*)]
\item\label{def:matroid:I1} $\varnothing \in \mathcal I$.
\item\label{def:matroid:I2} If $I \in \mathcal I$ and $I' \subseteq I$, then $I' \in \mathcal I$.
\item\label{def:matroid:I3} If $I_1$ and $I_2$ are in $\mathcal I$ and $|I_1| < |I_2|$, then there is an element $e$ of $I_2 \setminus I_1$
such that $I_1 \cup \{e\} \in \mathcal I$.
\end{enumerate}
If $\m$ is the matroid $(E, \mathcal I)$, then $\m$ is called a \emph{matroid on} $E$. We
shall often write $\mathcal I(\m)$ for $\mathcal I$ and $E(\m)$ for $E$. The members of $I$ are the \emph{independent sets} of $M$, and $E$ is the \emph{ground set} of $\m$. A subset of $E$ that is not in $\mathcal I$ is called \emph{dependent}.
\end{definition}

\begin{definition}[Cf.~{\cite[Proposition~1.1.1]{Oxley2011}}]\label{def:VecMat}Let $A \in \M_{m\,n}(\F)$. By $\m[A]$ we denote a matroid with $E$ being  the set of column labels of  $A$ and $\mathcal I$ being the set of subsets $X$ of $E$ for which the multiset of columns labelled by $X$ is a set and is linearly independent in the vector space $\F^m$. This matroid is called the \emph{vector matroid} of $A$.
\end{definition}

It is indeed a matroid by~\cite[Proposition~1.1.1]{Oxley2011}.

\begin{definition}[Cf.~{\cite[Definition on p.~15]{Oxley2011}}]Let $\m$ be a matroid. We call a maximal
independent set in $\m$ a \emph{basis} or a \emph{base} of $\m$. By $\mathcal B(\m)$ we denote the set of all bases of $\m$.
\end{definition}

\begin{definition}[Cf.~{\cite[Definition on p.~20]{Oxley2011}}]\label{def:RestrictionMatroid}
Let $\m$ be the matroid $(E, \mathcal I)$ and suppose that $X \subseteq E$. Let $\mathcal I|X = \{I \subseteq X \colon
I \in \mathcal I\}$. Then it is easy to see that the pair $(X, \mathcal I|X)$ is a matroid. We call this matroid the \emph{restriction of $\m$ to $X$} or the \emph{deletion of $E \setminus X$ from $\m$}. It is denoted by $\m|X$ or $\m\setminus (E\setminus X)$.
\end{definition}
\begin{definition}Let $\m$ be the matroid $(E, \mathcal I)$ and $X \subseteq E$. We define the \emph{rank} $r(X)$ of $X$ to be the cardinality of a basis $B$ of $\m|X$ and call such a set $B$ a \emph{basis of} $X$.

We often write $r$ as $r_{\m}$. In addition, we usually write $r(\m)$ for $r(E(\m))$.
\end{definition}

\begin{lemma}[{Cf.~\cite[Lemma 1.3.1 and text on the top of p.~21]{Oxley2011}}]\label{lem:matroidRankProp}The rank function $r$ of a matroid $\m$ on a set $E$ has the following
properties:
\begin{enumerate}[label=(R\arabic*), ref=(R\arabic*)]
\item\label{lem:matroidRankProp:R1} If $X \subseteq E$, then $0 \le r(X) \le |X|$.
\item\label{lem:matroidRankProp:R2} If $X \subseteq Y \subseteq E$, then $r(X) \le r(Y)$.
\item\label{lem:matroidRankProp:R3} If $X$ and $Y$ are subsets of $E$, then
\[
r(X \cup Y ) + r(X \cap Y ) \le r(X) + r(Y ).
\]
\end{enumerate}
\end{lemma}
\begin{lemma}[{See~\cite[Theorem 1.3.2]{Oxley2011}}]\label{lem:RankPropImplyMatroid}Let $E$ be a set and $r$ be a function that maps $2^E$ into the set of non-negative integers and satisfies \ref{lem:matroidRankProp:R1}–\ref{lem:matroidRankProp:R3}. Let $\mathcal I$ be the collection of subsets $X$ of $E$ for which $r(X) = |X|$. Then $(E, \mathcal I)$ is a matroid having rank function $r$.
\end{lemma}

\begin{definition}[{\cite[Definition on p.~65]{Oxley2011}}] Let $\m$ be a matroid. The matroid, whose ground set is $E(\m)$ and whose set of bases is $\{E(M)\setminus B \colon B \mbox{ is the basis of } \m\}$, is called the \emph{dual} of $\m$ and is denoted by $\m^{*}$.
\end{definition}

$\m^*$ is indeed a matroid by~\cite[Theorem~2.1.1]{Oxley2011}.

\begin{definition}[Cf.~{\cite[Definitions on p.~65 and p.~67]{Oxley2011}}]
The bases of $\m^{*}$ are called \emph{cobases} of $\m$. A similar convention applies to other distinguished subsets of $E(\m^*)$. Hence, for example, the independent sets and spanning sets of $\m^*$ are called coindependent sets and cospanning sets of $\m,$ and $r^*$ denotes $r_{\m^*}$ and is called a \emph{corank} function of $\m$.
\end{definition}

\begin{lemma}[Cf.~{\cite[Proposition~2.1.9]{Oxley2011}}]\label{lem:DualMatroidRank}For all subsets $X$ of the ground set $E$ of a matroid $\m$,
\begin{equation*}
r^*(X)+ r(E) = r(E\setminus X) + |X|.
\end{equation*}
\end{lemma}

\begin{lemma}[{See~\cite[Lemma~2.1.10]{Oxley2011}}]\label{lem:MatroidIndepCoindepSep}
Let $I$ and $I^*$ be disjoint subsets of $E(\m)$ such that $I$ is independent and $I^*$ is coindependent. Then $\m$ has a basis $B$ and a cobasis $B^*$ such that $B$ and $B^*$ are disjoint, $I \subseteq B$, and $I^* \subseteq B^*$.
\end{lemma}

\begin{definition}[Cf.~{\cite[Definition on p.~100]{Oxley2011}}]\label{def:ContractionMatroid}Let $\m$ be a matroid on $E$, and $T$ be a subset of $E$. Let $\m/T$, the \emph{contraction of $T$ from $\m$}, be given by
\[
\m/T = (\m^*\setminus T)^*.
\]
\end{definition}

\begin{lemma}[{C.f.~\cite[Proposition 3.1.6]{Oxley2011}}]\label{lem:ContractionRank}
If $T \subseteq E$, then, for all $X \subseteq E \setminus T$,
\[
 r_{\m/T}(X) = r_\m(X \cup T) - r_\m(T).
 \]
\end{lemma}

\begin{corollary}\label{cor:ContractionIndependent}Let $\m$ be a matroid on $E$ and $T$ be an independent set of $\m$. Then $X \supseteq T$ is an independent set of $M$ if and only if $X \setminus T$ is an independent set of $\m/T$. 
\end{corollary}

For purpose of the proof of the main theorem of this article we need to provide the correspondence between the families of coindependent sets and coranks of matroid $\m$ and contraction matroid $\m|S$.
%
%\begin{underconst}
%\begin{lemma}\label{lem:MatroidCobaseRestriction}Let $\m$ be a matroid on $E$, $S \subseteq E$ be a set and $I \subseteq E$ be an independent set of $\m$. Then
%\begin{enumerate}[label=(\alph*), ref=(\alph*)]
%\item\label{lem:MatroidCobaseRestriction:part1} If ${B'}$ is a basis of $\m/(E\setminus S)$, then $\left(I \setminus S\right) \cup {B'}^*$ is a spanning set of $\m$.
%\item If ${B'}$ is a basis of $\m/(E\setminus S)$, 
%\item ${B'}$ is a cobasis of $\m^*|S$,
%\item $\left(I^* \setminus S\right) \cup {B'}^*$ is a coindependent set of $\m^*$
%\item $\left(I^* \setminus S\right) \cup {B'}^*$ is a spanning set of $\m$
%
%then $\left(I^* \setminus S\right) \cup {B'}^*$ is a coindependent set of $\m$.
%\item\label{lem:MatroidCobaseRestriction:part2} $r^*(\m|S) \le r^*(\m) - \left|I^* \setminus S\right|.$
%\end{enumerate}
%\end{lemma}
%
%\begin{lemma}\label{lem:MatroidCobaseRestriction}Let $\m$ be a matroid on $E$, $S \subseteq E$ be a set and $I \subseteq E$ be a spanning set of $\m$. Then
%\begin{enumerate}[label=(\alph*), ref=(\alph*)]
%\item\label{lem:MatroidCobaseRestriction:part1} If ${B'}$ is a basis of $\m/S$, then $\left(I \cap S\right) \cup {B'}$ is a spanning set of $\m$.
%
%\item\label{lem:MatroidCobaseRestriction:part2} $r^*(\m|S) \le r^*(\m) - \left|I^* \setminus S\right|.$
%\end{enumerate}
%\end{lemma}
%\end{underconst}

\begin{lemma}\label{lem:MatroidCobaseRestriction}Let $\m$ be a matroid on $E$, $S \subseteq E$ be a set and $I^* \in \mathcal I(\m^*)$. Then
\begin{enumerate}[label=(\alph*), ref=(\alph*)]
\item\label{lem:MatroidCobaseRestriction:part1} If ${B'}^* \in \mathcal B(\m|S)$, then $\left(I^* \setminus S\right) \cup {B'}^* \in \mathcal I(\m)$.
\item\label{lem:MatroidCobaseRestriction:part2} $r^*(\m|S) \le r^*(\m) - \left|I^* \setminus S\right|.$
\end{enumerate}
\end{lemma}

\begin{proof}
\textbf{\ref{lem:MatroidCobaseRestriction:part1}}\quad Since ${B'}^*$ is a cobasis of $\m|S$, $B' = \left(S \setminus {B'}^*\right) \subseteq S$ is an independent set of $\m|S$. Hence, $B'$ is an independent set of $\m$ by the definition of $\m|S$. %$B' \subseteq S$.   

Let $I^*_S = I^* \setminus S \in \mathcal I(\m^*)$. Then from $I^* \in \mathcal I(\m^*)$ we have that $I^*_S \in \mathcal I(\m^*)$. In addition, $I^*_S \cap S = \varnothing$.

Thus, $I^*_S$ and $B'$ are correspondingly coindependent and independent subsets of $\m$  such that  $I^*_S \cap B' = \varnothing$. Hence, from Lemma~\ref{lem:MatroidIndepCoindepSep} we conclude that there exist a basis $B_0$ and a cobasis $B_0^*$ of $\m$ such that  $B_0 \cap B_0^* = \varnothing$, $B_0 \supseteq B'$ and 
\begin{equation}\label{MatroidCobaseRestriction:eq101}
B_0^* \supseteq I^*_S.
\end{equation}

The inclusion $B_0 \supseteq B'$ implies the inclusion 
\begin{equation}\label{MatroidCobaseRestriction:eq1}
B_0 \cap S \supseteq B' \cap S = B'.
\end{equation}

The definition of $\m|S$ implies that $B_0 \cap S$ is an independent set of $\m|S$. Thus, having the inclusion~\eqref{MatroidCobaseRestriction:eq1} and the fact that $B'$ is base of $\m|S$, we obtain that 

\begin{equation}\label{MatroidCobaseRestriction:eq2}
B_0 \cap S = B'.
\end{equation}
Since $E = B_0 \sqcup B_0^*$ and ${B'} \sqcup {B'}^* = S \subseteq E$, it follows from~\eqref{MatroidCobaseRestriction:eq2} that $B_0^* \cap S = {B'}^*$. Hence, 
\begin{equation}\label{MatroidCobaseRestriction:eq102}
B_0^* \supseteq {B'}^*.
\end{equation}
From~\eqref{MatroidCobaseRestriction:eq101} and~\eqref{MatroidCobaseRestriction:eq102} we conclude that
$$B_0^* \supseteq I^*_S\cup {B'}^* = \left(I^* \setminus S\right) \cup {B'}^*.$$
Therefore, $\left(I^* \setminus S\right) \cup {B'}^*$ is a coindependent set of $\m$. 

\paragraph{\ref{lem:MatroidCobaseRestriction:part2}} Let ${B'}^* \subseteq S$ be any cobasis of $\m|S$. Then $r^*\left(\m|S\right) = \left|{B'}^*\right|$ by the definition of the rank function. From the part~\ref{lem:MatroidCobaseRestriction:part1} of the lemma we conclude that  $\big(\left(I^* \setminus S\right) \cup {B'}^*\big) \in \mathcal I(m^*)$ and consequently
\begin{equation*}
r^*(\m) \ge \left|\left(I^* \setminus S\right) \cup {B'}^*\right|.
\end{equation*}
Let us consider the right hand side of this inequality. Since $\left(I^* \setminus S\right) \cap {B'}^*  = \varnothing$,
\begin{equation*} 
\left|\left(I^* \setminus S\right) \cup {B'}^*\right|
= \left|I^* \setminus S\right| + \left|{B'}^*\right| = \left|I^* \setminus S\right| + r^*\left(\m|S\right).
\end{equation*}
Thus,
\[
r^*(\m) \ge \left|\left(I^*\right) \setminus S\right| + r^*(\m|S).
\]
which is equivalent to the required inequality.
\end{proof}

The following corollary is a direct consequence of the definition of contraction matroid and Lemma~\ref{lem:MatroidCobaseRestriction} applied to the contraction matroid $\m/T$ instead of $\m$ and $I^*\setminus T$ instead of $I^*$. It will be used in sequel.

\begin{corollary}\label{lem:MatroidCobaseContractionRestriction}Let $\m$ be a matroid on $E$, $T \subseteq E$, $S \subseteq E \setminus T$ be two sets and $I^* \mathcal I(\m^*)$. Then
\begin{enumerate}[label=(\alph*), ref=(\alph*)]
\item\label{lem:MatroidCobaseContractionRestriction:part1} If ${B'}^* \in \mathcal B((\m/T|S)^*)$, then $\big(\left(\left(I^* \setminus T\right) \setminus S\right) \cup {B'}^*\big) \in \mathcal I(\m^*)$.
\item\label{lem:MatroidCobaseContractionRestriction:part2} $r^*\left(\m/T|S\right) \le r^*\left(\m\right) - \left|\left(I^* \setminus T\right) \setminus S\right|.$
\end{enumerate}
\end{corollary}

\section{Matroids corresponding to vector spaces and linear varieties}\label{sec:matroidlinvar}

In this section we introduce a matroid $\m(\K)$ corresponding to a given linear variety $\K$ (Definition~\ref{def:MatroidMK}). After that, we study the behaviour of $\m(\K)$ while $\K$ is transformed using the projection map or intersected with a slice of $\F^S$ (Definition~\ref{def:ProjectionMap} and Definition~\ref{def:Slice}). The corresponding results are stated and proved in Lemma~\ref{lem:ALSLifting} and Lemma~\ref{lem:LinVarSecProj}. We also prove that a matroid linear variety and vector matroid of a matrix in its equational representation are dual to each other (Lemma~\ref{lem:IdepMatroidCorr}).%\bigskip

\begin{definition}\label{def:matSubspace}Let $V \subseteq \F^E$ be a vector space. Then by $\m(V)$ we denote a matroid determined by rank function $r(X) = \dim \left(\Span\left(\x^E_i|_V \mid i \in X\right)\right)$.
\end{definition}

Since $r$ satisfies the properties~\ref{lem:matroidRankProp:R1}--\ref{lem:matroidRankProp:R3}, $\m(V)$ is indeed a matroid by Lemma~\ref{lem:RankPropImplyMatroid}.

\begin{remark}Every matroid of a subspace $V\subseteq F$ defined above can be regarded as an example of a matroid \emph{representable over} $\F$ (see~\cite[Chap.~9.1, definition on p.~136]{welsh_matroid_1986}).
\end{remark}

\begin{definition}\label{def:matLinVar}\label{def:MatroidMK}Let $\K \subseteq \F^E$ be a linear variety. Then by $\m(\K)$ we denote \emph{a matroid of the linear variety} $\K$ which is defined by $\m(\K) = \m(V)$, where $V \subseteq F^E$ is the vector space belonging to $\K$. 
\end{definition}

\begin{corollary}\label{cor:matLinVarDim}If $\K$ is a linear variety, then $r(\m(\K)) = \dim(\K)$ and $r^*(\m(\K)) = \codim(\K)$.
\end{corollary}

\begin{lemma}\label{lem:IndepEveryValVectors}Let $V \subseteq \F^E$ be a vector space, $X \subseteq E$. Then a linear map $\pi_E^X|_V$ is a surjection if and only if $X$ is an independent set of $\m(V)$.
\end{lemma}

\begin{proof}Let  $\phi \colon \left(\F^X\right)^* \to V^*$ be defined by  $\phi(f) = \sum_{i \in X} f(\mathbf e_i) \x_i|_V$. Note that 
\begin{equation}\label{lem:XndepEveryValVectors:eq1}
\phi = \left(\pi_E^X|_V\right)^*.
\end{equation}
Indeed, since $\x_{i}^X \circ \pi_E^X|_V = \x_{i}^E|_V$ for all $i \in X$, then $\left(\pi_E^X|_V\right)^*\left(\x_i^X\right) = \phi(\x_i^X)$ for all $i \in X$. By linearity, the equality~\eqref{lem:XndepEveryValVectors:eq1} holds since it holds on the basis $\{\x_i^X\;\mid\; i \in X\}$ of $\left(\F^X\right)^*$.

Using the basic properties of dual vector spaces (see~\cite[Sec.~2.7, Exercise~20 on p.~127]{Friedberg2002-wu}) we conclude from~\eqref{lem:XndepEveryValVectors:eq1}  that 
\begin{equation}\label{lem:XndepEveryValVectors:eq2}
\pi_E^X|_V\;\;\mbox{is a surjection} \Leftarrow\joinrel=\joinrel\Rightarrow \phi\;\;\mbox{is an injection}.
\end{equation} 

Since $\phi(\x^X_i) = \x^E_i|_V$ for all $i \in X$ by the definition of $\phi$ and the set $\{\x^X_i \mid i \in X\}$ is a basis of $\left(\F^X\right)^*$, then
\begin{equation}\label{lem:XndepEveryValVectors:eq3}
\phi\;\;\mbox{is an injection} \Leftarrow\joinrel=\joinrel\Rightarrow \dim \left(\Span\left(\x^E_i|_V \mid i \in X\right)\right) = |X|.
\end{equation} 

The definition of $\m(V)$ implies that 
\begin{equation}\label{lem:XndepEveryValVectors:eq4}
\dim \left(\Span\left(\x^E_i|_V \mid i \in X\right)\right) = |X|  \Leftarrow\joinrel=\joinrel\Rightarrow X \in \mathcal I(\m(V)) .
\end{equation} 

Thus, the required equivalence is obtained by aligning the equivalences~\eqref{lem:XndepEveryValVectors:eq2}--\eqref{lem:XndepEveryValVectors:eq4} together.
\end{proof}

\begin{corollary}\label{cor:IndepEveryVal}Let $\K \subseteq \F^E$ be a linear variety. Then $I \in \mathcal I(\m(\K))$ if and only if for every $\mathbf c \in \F^I$ there exists $\mathbf x \in \K$ such that $\mathbf x_i = \mathbf c_i$ for all $i \in I$.
\end{corollary}

Let us study the behaviour of $\m(\K)$ under projections and restrictions. The next lemma provide a correspondence between a matroid of linear variety and a matroid of its $f$-projection.

\begin{lemma}\label{lem:ALSLifting}
Let $E_1$ and $E_2$ be two finite sets, $f\colon E_2 \to E_1$ be an injective function, $\K \subseteq \F^{E_1}$ be a linear variety. Then
\begin{enumerate}[label=(\alph*), ref=(\alph*)]
\item\label{lem:ALSLifting:part4} $r_{\m(\pi_f(\K))}(X) = r_{\m(\K)}(f(X))$ for all $X \subseteq E_2$;
\item\label{lem:ALSLifting:part2} $f$ is an isomorphism of matroids $\m(\pi_f(\K))$ and $\m(\K)|f(E_2)$ (restriction matroid);
\end{enumerate}
\end{lemma}
\begin{proof}\textbf{\ref{lem:ALSLifting:part4}}\quad Let $V\subseteq \F^{E_1}$ be vector space belonging to $\K$. Then Lemma~\ref{lem:LinImageAS} implies that $\pi_f(V)$ is a vector space belonging to $\pi_f(\K)$. Let us consider $\pi_f|_V$ as a linear map on its image, that is, we assume that $\pi_f|_V$ sends $V$ on $\pi_f(V)$. Note that this together with the definition of $\pi_f$ implies that 
\begin{equation}\label{lem:ALSLifting:eq1}
\left(\pi_f|_V\right)^* (\x_i^{E_2}|_{\pi_f(V)}) = \x_{f(i)}^{E_1}|_V\;\;\mbox{for all}\;\; i \in E_2.
\end{equation}

Since $\pi_f|_V$ is a surjection, $\left(\pi_f|_V\right)^*$ is an injection. Hence, using~\eqref{lem:ALSLifting:eq1} we conclude that
\begin{equation}\label{lem:ALSLifting:eq3}
\dim \left(\Span\left(\x_i^{E_2}|_{\pi_f(V)}\;\mid\; i \in X\right)\right) = \dim \left(\Span\left(\x_{f(i)}^{E_1}|_V\;\mid\; i \in f(X)\right)\right)\;\;\mbox{for all}\;\;X \subseteq E_2.
\end{equation}
Therefore,
\begin{multline}\label{lem:ALSLifting:eq2}
r_{\m(\pi_f(\K))}(X) = \dim \left(\Span\left(\x_i^{E_2}|_{\pi_f(V)}\;\mid\; i \in X\right)\right)\\
  \overset{\eqref{lem:ALSLifting:eq3}}{=\joinrel=\joinrel=}\dim \left(\Span\left(\x_{f(i)}^{E_1}|_V\;\mid\; i \in f(X)\right)\right) = r_{\m(\K)}(f(X))\;\;\mbox{for all}\;\;X \subseteq E_2
\end{multline}
by the definition of the rank function.

\paragraph{\ref{lem:ALSLifting:part2}} It follows directly from the part~\ref{lem:ALSLifting:part4} of the lemma and the definition of restriction matroid.
\end{proof}

\begin{definition}Let $E_1$ and $E_2$ be two finite sets such that $E_2 \subseteq E_1$, and $V \subseteq \F^{E_1}$ be a vector space. Then by $V_{E_2} \subseteq \F^{E_1}$ we denote a vector space defined by 
$$V_{E_2} = \{\mathbf x \in V\,\mid\, \mathbf x_i = 0\;\mbox{for all}\;\; i \in E_2\}.$$
\end{definition}

\begin{lemma}\label{lem:VectorEmbeddingMatroid}Let $E_1$ and $E_2$ be two finite sets such that $E_2 \subseteq E_1$, and $V \subseteq \F^{E_1}$ be a vector space. Then
\begin{enumerate}[label=(\alph*), ref=(\alph*)]
\item\label{lem:VectorEmbeddingMatroid:part1} 
\begin{equation}\label{lem:VectorEmbeddingMatroid:eq3}
r_{\m(V_{E_2})}(X) = r_{\m(V)}\left(X\cup E_2\right) - r_{\m(V)}\left(E_2\right)\;\;\mbox{for all}\;\;X \subseteq \left(E_1 \setminus E_2\right)
\end{equation}
\item\label{lem:VectorEmbeddingMatroid:part2} $\m(V_{E_2})\setminus E_2 = \m(V)/E_2$.
\end{enumerate}
\end{lemma}
\begin{proof}\textbf{\ref{lem:VectorEmbeddingMatroid:part1}}\quad The definition of $V_{E_2}$ implies that 
\begin{equation}\label{lem:VectorEmbeddingMatroid:eq2}
\x_{i}^{E_1}|_{V_{E_2}} = 0\;\;\;\;\mbox{for all}\;\;i \in E_2
\end{equation}
Let $X \subseteq \left(E_1 \setminus E_2\right)$ and let $\phi \colon \Span\left(\x_i^{E_1}|_{V}\,\mid\, i \in \left(X \cup E_2\right)\right) \to \Span\left( \x_i^{E_1}|_{V_{E_2}}\,\mid\, i \in X\right)$ be a linear map defined by
\[
\phi\left(\sum_{i \in \left(X \cup E_2\right)}\lambda_i \x_i^{E_1}|_{V}\right) = \sum_{i \in X}\lambda_i \x_i^{E_1}|_{V_{E_2}}.
\] 
The equality~\eqref{lem:VectorEmbeddingMatroid:eq2} implies that $\phi$ is defined properly, i.e. does not depend of the representation of an element of $\Span\left(\x_i^{E_1}|_{V}\,\mid\, i \in \left(X \cup E_2\right)\right)$ as a linear combination $\sum_{i \in \left(X \cup E_2\right)}\lambda_i \x_i^{E_1}|_{V}$. Therefore, the following sequence is exact
\begin{multline*}
0 \to \Span\left(\x_i^{E_1}|_{V}\,\mid\, i \in E_2\right)  \hookrightarrow \Span\left(\x_i^{E_1}|_{V}\,\mid\, i \in \left(X \cup E_2\right)\right)\\
\overset{\phi}{\longrightarrow} \Span\left(\x_i^{E_1}|_{V_{E_2}}\,\mid\, i \in X\right) \to 0.
\end{multline*}
Hence,
\begin{multline}\label{lem:VectorEmbeddingMatroid:eq2}
\dim\left(\Span\left(\x_i^{E_1}|_{V}\,\mid\, i \in E_2\right)\right) + \dim\left(\Span\left(\x_i^{E_1}|_{V_{E_2}}\,\mid\, i \in X\right)\right)\\
= \dim\left(\Span\left(\x_i^{E_1}|_{V}\,\mid\, i \in \left(X \cup E_2\right)\right)\right).
\end{multline}
By the definition of the corresponding rank functions,
\begin{multline}\label{lem:VectorEmbeddingMatroid:eq3}
\dim\left(\Span\left(\x_i^{E_1}|_{V}\,\mid\, i \in E_2\right)\right) = r_{\m(V)} \left(E_2\right),\\ \dim\left(\Span\left(\x_i^{E_1}|_{V_{E_2}}\,\mid\, i \in X\right)\right) = r_{\m(V_{E_2})}(X)\quad\mbox{and}\\
\dim\left(\Span\left(\x_i^{E_1}|_{V}\,\mid\, i \in \left(X \cup E_2\right)\right)\right) = r_{\m(V)}\left(X \cup E_2\right).
\end{multline}
By substituting~\eqref{lem:VectorEmbeddingMatroid:eq3} into~\eqref{lem:VectorEmbeddingMatroid:eq2} we conclude that
\[
 r_{\m(V)} \left(E_2\right) +   r_{\m(V_{E_2})}(X) = r_{\m(V)}\left(X \cup E_2\right)
\]
which is equivalent to the equality~\eqref{lem:VectorEmbeddingMatroid:eq3}.

\paragraph{\ref{lem:VectorEmbeddingMatroid:part2}}This follows directly from Lemma~\ref{lem:ContractionRank}, the part~\ref{lem:VectorEmbeddingMatroid:part1} of the lemma and the definitions of the corresponding matroids.
%\paragraph{\ref{lem:ALSIntersectsSlice:part4}} This is a dual version of \ref{lem:ALSIntersectsSlice:part1}.
\end{proof}

\begin{corollary}\label{cor:DimEmbRank}Let $E$ be a finite set, $X \subseteq E$ and $V \subseteq \F^{E}$ be a vector space. Then $$\dim \left(V_{E\setminus X}\right) = r\left(\m(V)\right) - r_{\m(V)}(E\setminus X).$$
\end{corollary}

In order to generalize Lemma~\ref{lem:VectorEmbeddingMatroid} to linear varieties we need to introduce a notion of slice.

\begin{definition}\label{def:Slice}For $\{e_1,\ldots, e_r\} = E' \subseteq E$ and $\mathbf c' \in \F^{E'}$ we call a \emph{slice} $U(x_{a_1} = \mathbf c_{e_1}, \ldots, x_{e_r} = \mathbf c_{e_r})$ a linear variety of $\F^E$ defined by
\[
U(x_{e_1} = \mathbf c_{e_1}, \ldots, x_{e_r} = \mathbf c_{e_r}) = \mathbf c + \Span\left(\mathbf e_i\;\mid\; i \in E \setminus E'\right),
\]
where $\mathbf c_i = \mathbf c'_i$ for $i \in E'$ and is equal to zero otherwise. Equivalently,

\[
U(x_{e_1} = \mathbf c_{e_1}, \ldots, x_{e_r} = \mathbf c_{e_r}) = \{\mathbf s \in \F^E \mid \mathbf s_e = \mathbf c_e\; \forall e \in E'\}.
\]

We will also use a shorter notation $U(x_{e} = \mathbf c_e\;|\; e \in E')$.
\end{definition}

\begin{corollary}\label{cor:SliceInZero}If $E' \subseteq E$ and $\mathbf c' \in \F^{E'}$, then  $U(x_{e} = 0\;|\; e \in E')$ is a vector space belonging to $U(x_{e} = \mathbf c_e\;|\; e \in E')$.
\end{corollary}

\begin{lemma}\label{lem:SliceBelong}Let $\K \subseteq \F^E$ be a linear variety, $V$ be a subspace belonging to $\K$, $I$ be an independent set of $\m(\K)$, $\mathbf c \in \F^I$ and $\K_I = \K\cap U(x_{e} = \mathbf c_e\,|\,e \in I)$. Then $\K_I$ is a linear variety and 
$V_I$ is a subspace belonging to $\K_I$.
\end{lemma}
\begin{proof}Lemma~\ref{lem:ASIntersec} implies that in order to prove that $\K_I$ is a linear variety it is sufficient to show that $\K$ and $U(x_{e} = \mathbf c'_e\,|\,e \in I)$ have non-empty intersection. This follows directly from Corollary~\ref{cor:IndepEveryVal}. 

Next, since $U(x_{e} = 0 | e \in I)$ is a subspace belonging to $U(x_{e} = \mathbf c_e | e \in I)$ by Corollary~\ref{cor:SliceInZero}, then $V \cap U(x_{e} = 0\,|\,e \in I)$ is a subspace belonging to $\K_I$ by Lemma~\ref{lem:ASIntersecVecBelong}. In addition, $V \cap U(x_{e} = 0\,|\,e \in I) =  V_I$ by the definition of $V_I$.
\end{proof}

\begin{remark}If $I$ is not an independent set of $\m(\K)$, then the intersection $\K\cap U(x_{e} = \mathbf c_e\,|\,e \in I)$ could be empty. %However, if we assume that $I$ is is an independent set of
\end{remark}

In the next lemma we provide a correspondence between a matroid of linear variety $\K$ and a matroid of its intersection with slice $U(x_{e} = \mathbf c_e\,|\,e \in I)$ assuming that $I$ is an independent set of $\K$ which follows directly from the correspondence established in Lemma~\ref{lem:SliceBelong}. This lemma together with its generalized version (Lemma~\ref{lem:LinVarSecProj}) will be used in the next section.

\begin{lemma}\label{lem:LinVarSec}
Let $E$ be a finite set, $\K \subseteq \F^{E}$ be a linear variety, $I$ be an independent set of $\m(\K)$, $\mathbf c \in \F^{I}$. Then $\K_I = \K\cap U(x_{e} = \mathbf c_e | e \in I)$ is a linear variety and $\m(\K_{I})\setminus I = \m(\K)/I$. 
\end{lemma}
\begin{proof}
Lemma~\ref{lem:SliceBelong} implies that $\K_I$ is indeed a linear variety and $V_I$ is a vector space belonging to $\K_I$, where $V$ is a vector space belonging to $\K$. Therefore, by Lemma~\ref{lem:VectorEmbeddingMatroid}\ref{lem:VectorEmbeddingMatroid:part2} applied to $I, E$ and $V$, we obtain the statement of the lemma.  
\end{proof}

\begin{lemma}\label{lem:LinVarSecProj}
Let $E_1, E_2$ be finite sets, $\K \subseteq \F^{E_1}$ be a linear variety, $I$ be an independent set of $\m(\K)$, $\mathbf c \in \F^{I}$,   $f\colon E_2 \to E_1$ be an injective function. Then  
$f$ provides an isomorphism between the matroids $\m(\pi_f(\K_I))\setminus f^{-1}(I)$ and $(\m(\K)/I)|\left(f(E_2)\setminus I\right)$, where $\K_I$ is defined by $\K_I = \K\cap U(x_{e} = \mathbf c_e | e \in I)$ 
\end{lemma}
\begin{proof}
Lemma~\ref{lem:LinVarSec} applied to $E_1, \K$ and $I$ implies that 
\[
\m(\K_{I})\setminus I = \m(\K)/I.
\]
Hence,
\[
(\m(\K_{I})\setminus I) | (f(E_2)\setminus I) = (\m(\K)/I) | (f(E_2)\setminus I).
\]
In addition, since $(E_1 \setminus I) \supseteq (f(E_2) \setminus I)$, 
\[
(\m(\K_{I})\setminus I) | (f(E_2)\setminus I) = \m(\K_{I}) | (f(E_2)\setminus I).
\]
Therefore,
\begin{equation}\label{lem:LinVarSecProj:eq1}
\m(\K_{I}) | (f(E_2)\setminus I) = (\m(\K)/I) | (f(E_2)\setminus I).
\end{equation}

By Lemma~\ref{lem:ALSLifting}\ref{lem:ALSLifting:part2}, $f$ is an isomorphism between
$\m(\pi_f(\K_{I}))$ and $\m(\K_I|f(E_2))$. Hence, $f|_{(E_2\setminus f^{-1}(I))}$ is an isomorphism between $\m(\pi_f(\K_{I}))\setminus f^{-1}(I)$ and $\m\left(\K_I\right)|\left(f(E_2)\setminus I\right).$ This fact together with~\eqref{lem:LinVarSecProj:eq1} implies that $f$ provides an isomorphism between $\m(\pi_f(\K_{I}))\setminus f^{-1}(I)$ and $(\m(\K)/I) | (f(E_2)\setminus I).$
\end{proof}

The following lemma provides a simple correspondence between a matroid of linear variety $\K$ and a vector matroid of matrix in equational representation of $\K$ which will be used in sequel and has an independent importance.

\begin{lemma}\label{lem:IdepMatroidCorr}Let $\K = \AS(A,\mathbf b) \subseteq \F^E$ be a linear variety for some $A \in \M_{m\, E}(\F), \mathbf b \in \F^n$. Then
\begin{enumerate}[label=(\alph*), ref=(\alph*)]
\item\label{lem:IdepMatroidCorr:part1} 
\begin{equation}\label{lem:IdepMatroidCorr:eq}
r_{\m[A]}(X) + r(\m(\K)) = |X| + r_{\m(\K)}\left(E\setminus X\right)\;\;\mbox{for all}\;\;X \subseteq E.
\end{equation}
\item\label{lem:IdepMatroidCorr:part2}  $\m[A] = \m^*(\K)$.
\end{enumerate}
\end{lemma}
\begin{proof}
\textbf{\ref{lem:IdepMatroidCorr:part1}}\quad Let $V \subseteq \F^E$ be a vector space belonging to $\K$. It follows from Corollary~\ref{cor:SolSetLinVty} that 
\begin{equation}\label{lem:IdepMatroidCorr:eq2}
V = \AS(A,\mathbf 0).
\end{equation}

Let $X \subseteq E$. Consider the following sequence of vector spaces and linear maps
\begin{equation}\label{lem:IdepMatroidCorr:eq1}
0 \to V_{E\setminus X} \overset{\pi_{E}^X}{\relbar\joinrel\relbar\joinrel\relbar\joinrel\relbar\joinrel\longrightarrow} \F^X
 \overset{\mathbf x \mapsto \sum_{i \in X} \mathbf x_i A(|i]}{\relbar\joinrel\relbar\joinrel\relbar\joinrel\relbar\joinrel\relbar\joinrel\relbar\joinrel\relbar\joinrel\relbar\joinrel\longrightarrow} \colsp \left(A(|X]\right) \to 0,
\end{equation}
where $\colsp \left(A(|X]\right)$ is a column space of a matrix $A(|X]$. The definition of $V_{E\setminus X},$ $\pi_{E}^X$ and the equality~\eqref{lem:IdepMatroidCorr:eq2} imply that this sequence is exact. Hence,
\begin{equation}\label{lem:IdepMatroidCorr:eq3}
\dim(\F^X) = \dim\left(V_{E\setminus X}\right) + \dim\left(\colsp \left(A(|X]\right)\right).
\end{equation}

By the definition of the rank function,
\begin{equation}\label{lem:IdepMatroidCorr:eq5}
r\left(\m(V)\right) = r_{\m(\K)}(E)
\end{equation}
Hence, using Corollary~\ref{cor:DimEmbRank} we conclude that
\begin{equation}\label{lem:IdepMatroidCorr:eq4}
\dim\left(V_{E\setminus X}\right) = r\left(\m(V)\right) - r_{\m(V)}(E\setminus X) \overset{\eqref{lem:IdepMatroidCorr:eq5}}{=\joinrel=\joinrel=}  r_{\m(\K)}(E) - r_{\m(V)}(E\setminus X).
\end{equation}

The definition of $\m[A]$ implies that 
\begin{equation}\label{lem:IdepMatroidCorr:eq6}
\dim\left(\colsp \left(A(|X]\right)\right) = r_{\m[A]}(X).
\end{equation}

Thus, by substituting~\eqref{lem:IdepMatroidCorr:eq4} and~\eqref{lem:IdepMatroidCorr:eq6} into  \eqref{lem:IdepMatroidCorr:eq3} we obtain that
\begin{equation}\label{lem:IdepMatroidCorr:eq7}
\dim(\F^X) =  r\left(\m(V)\right) - r_{\m(\K)}(E\setminus X) + r_{\m[A]}(X).
\end{equation}
Since $\dim(\F^X) = |X|$, the equality~\eqref{lem:IdepMatroidCorr:eq7} is equivalent to the equality~\eqref{lem:IdepMatroidCorr:eq}.

\paragraph{\ref{lem:IdepMatroidCorr:part2}} This follows directly from the part~\ref{lem:IdepMatroidCorr:part1} of the lemma and Lemma~\ref{lem:DualMatroidRank}.
\end{proof}

\begin{corollary}\label{cor:MatMatDualToIndepMat}Let $\K = \AS(A,\mathbf b) \subseteq \F^E$ be a linear variety. Then $\m(\K) = \m^*[A]$.
\end{corollary}

\begin{corollary}\label{cor:EqEqRepImplyEqMat}If $\varnothing \neq \AS(A,\mathbf b) = \AS(A',\mathbf b')$, then $\m[A] = \m[A']$.
\end{corollary}

The following definition and two lemmas are used in the proof of Lemma~\ref{lem:ALSDetNKZeroDimKThenExistVec} in the next section.

\begin{definition}Let $\K = \AS(A,\mathbf b)$ be a linear variety and $B^* \subseteq E$ be a cobasis of $\m[A]$. Then $A$ is called \emph{reduced with respect to $B^*$} if $A$ has full rank and $A(|B^*]$ is a permutation matrix.
\end{definition}

\begin{lemma}\label{lem:ReducedERExists}Let $\K \subsetneq \F^E$ be a linear variety and $B^* \subseteq E$ be a cobasis of $\m(\K)$. Then there exists an equational representation $(A,\mathbf b)$ of $\K$ such that $A$ is reduced with respect to $B^*$.
\end{lemma}
\begin{proof}
Lemma~\ref{lem:ForLinVarExistsEqRep} implies that there is $A'' \in \M_{m\, E}(\F)$ and $\mathbf b'' \in \F^m$ such that  $\K = \AS(A'',\mathbf b'').$ By Lemma~\ref{lem:ChooseFullRank}, there exists $R \subseteq [m]$ such that $\K = \AS\Bigl(A''[R|),\mathbf b''[R|)\Bigr)$ and $A' = A''[R|)$ has full rank. For convenience, let $\mathbf b'$ denote the vector $\mathbf b''[R|)$. Thus,
\begin{equation}\label{lem:ReducedERExists:eq3}
\K = \AS(A',\mathbf b').
\end{equation}

Since $B^*$ is a cobasis of $\m(\K)$, it is a basis of $\m[A']$ by Lemma~\ref{lem:IdepMatroidCorr}\ref{lem:IdepMatroidCorr:part2}. 
Hence, \[
\rk\Bigl(A'(|B^*]\Bigr) = \left|B^*\right| = \rk(A').
\]
Therefore, $A'(|B^*]$ is a square matrix having full rank, which implies that it is invertible. Let us denote its inverse by $C$. Then
\begin{equation}\label{lem:ReducedERExists:eq4}
(CA')(|B^*] = C(A'(|B^*]) = I_{\codim(\K)}.
\end{equation}

By Lemma~\ref{lem:InvMultEquiv}
\begin{equation}\label{lem:ReducedERExists:eq5}
\AS(A,\mathbf b) = \AS(A',\mathbf b'),
\end{equation}
where  $A = CA'$ and $\mathbf b = C\mathbf b'$. By aligning~\eqref{lem:ReducedERExists:eq3} and~\eqref{lem:ReducedERExists:eq5} together we conclude that $\K = \AS(A,\mathbf b)$. In addition, \eqref{lem:ReducedERExists:eq4} implies that $A(|B^*]$ is an identity matrix and consequently $A$ is reduced with respect to $B^*$. Thus, the pair $(A, \mathbf b)$ satisfy all the conditions required in the statement of the lemma.
\end{proof}

\begin{lemma}\label{lem:ALSReducedBNonzeroExchange}Let $\K \subseteq \F^E$ be a linear variety, $B^* \subseteq E$ be a cobasis of $\m(\K)$, $(A,\mathbf b)$ be an equational representation of $\K$ such that $A$ is reduced with respect to $B^*$, $1 \le i \le \codim(\K)$, $e_{\mathrm{old}} \in B^*$ be a unique element such that $A_{i\, e_{\mathrm{old}}} \neq 0$. If $A_{i\, e_{\mathrm{new}}} \neq 0$
 for some $e_{\mathrm{new}} \in E$, then $B^* \triangle \{e_{\mathrm{old}},e_{\mathrm {new}}\}$ is a cobasis of $\m(\K)$.
\end{lemma}
\begin{proof}Indeed, in this case $A(|B^* \triangle \{e_{\mathrm{old}},e_{\mathrm {new}}\}]$ has the form
\[
\begin{pmatrix}
A(i|B^* \setminus e_{\mathrm{old}}\}] & *\\
0 & A_{i\, e_{\mathrm{new}}}
\end{pmatrix}
\]
after an appropriate rearrangement of rows and columns. It follows from the definition of $e_{\mathrm{old}}$ that $A(i|B^*\setminus\{e_{\mathrm{old}}\}]$ is a permutation matrix and hence it has full rank. Hence, $A(|B^* \triangle \{e_{\mathrm{old}},e_{\mathrm {new}}\}]$ is a square matrix having full rank. Therefore
$$\rk\Bigl(A(|B^* \triangle \{e_{\mathrm{old}},e_{\mathrm {new}}\}]\Bigr) = \left|B^* \triangle \{e_{\mathrm{old}},e_{\mathrm {new}}\}\right| = \left|B^*\right| = \rk(A)$$
which implies that $B^* \triangle \{e_{\mathrm{old}},e_{\mathrm {new}}\}$ is a basis of $\m[A]$.
\end{proof}

\section{Preliminaries from the theory of the Cullis' determinant}\label{sec:prelimcul}

In order to apply the results obtained above to the theory of the Cullis’ determinant, as presented in Section~\ref{sec:linvarsann}, we introduce several additional notations, definitions, and preliminary facts. The first six definitions are given following~\cite{NAKAGAMI2007422}.

\begin{definition}\label{def:NAKAG1}
By $[n]$ we denote the set $\{1, \ldots, n\}$.
\end{definition}
\begin{definition}
 By $\mathcal C_{X}^{k}$ we denote the set of injections from $[k]$ to $X$.
\end{definition}
\begin{definition}By $\binom{X}{k}$ we denote the set of the images of injections from $[k]$ to $X$, i.e. the set of all subsets of $X$ of cardinality $k.$
\end{definition}

\begin{definition}Suppose that $c \in \binom{[n]}{k}$ equal to $\{i_1,\ldots, i_k\}$, where $i_1 < i_2 < \ldots < i_k$, and $1 \le \alpha \le k$ is a natural number. Then $c(\alpha)$ is defined by
$$c(\alpha) = i_{\alpha}.$$
\end{definition}

\begin{definition}Given a set $c \in \binom{[n]}{k}$ we denote by $\sgn (c) = \sgn_{[n]} (c)$ the number
$$(-1)^{\sum_{\alpha = 1}^{k} (c(\alpha) - \alpha)}.$$
\end{definition}

\begin{note}$\sgn_{[n]}(c)$ depends only on $c$ and does not depend on $n.$
\end{note}

\begin{definition}\label{def:NAKAG2}Given an injection $\sigma \in \mathcal C_{[n]}^{k}$ we denote by $\sgn_{n\,k} (\sigma)$ the product $\sgn(\pi_\sigma) \cdot \sgn_{[n]} (c),$
where $\sgn(\pi)$ is the sign of the permutation
$$
\pi_\sigma =
\begin{pmatrix}
i_1 & \ldots & i_k\\
\sigma(1) & \ldots & \sigma(k)
\end{pmatrix},
$$
where $\{i_1, \ldots, i_k\} = \sigma([k])$ and $i_1<i_2<\ldots <i_k$.
\end{definition}

We omit the subscripts for $\sgn_{[n]}$ and $\sgn_{n\,k}$ if this cannot lead to a misunderstanding.%\bigskip

\begin{definition}[{Cf.~\cite[Corollary~2.6]{Guterman2025}}]\label{def:CullisDet}%\label{lem:DetNKAsSumProj}\label{lem:CullisAsSumDet}
Let  $n \ge k\ge 1$ be integers, $X \in \M_{n\,k} (\mathbb F)$ be a matrix. Then Cullis' determinant $\det_{n\, k}(X)$ of $X$ is defined by
\begin{equation*}
\det_{n\,k}(X) = (-1)^{(1 + \ldots + k)}\sum_{1 \le c_1 < \ldots < c_k \le n } (-1)^{c_1 + \ldots + c_k} \begin{vmatrix}x_{c_1\, 1} & \cdots & x_{c_1\, k}\\ \vdots & \ddots & \vdots\\ x_{c_k\, 1} & \cdots & x_{c_k\, k}\end{vmatrix}.
\end{equation*}
That is, the Cullis' determinant of $X$ is an alternating sum of basic minors of $X$. We also denote $\det_{n\, k} (X)$ as follows
\[
\det_{n\,k}(X)=\begin{vmatrix}X_{1\,1} & \cdots & X_{1\,k}\\
\vdots & \cdots & \vdots\\
X_{n\,1} & \cdots & X_{n\,k}
\end{vmatrix}_{n\,k}.
\]
In the case if $n = k$, then $\det_{n\,n}$ is also denoted as $\det_{k}$ and is clearly equal to the ordinary determinant of a square matrix.
\end{definition}

Lemma below provides a definition for the Cullis' determinant using another notation. This definition will be used in the sequel as well as Definition~\eqref{def:CullisDet}.

\begin{lemma}\label{lem:DetNKAsSumProj}\label{cor:CullisAsSumDet}
Let $X\in \M_{n\, k}(\F)$. Then
\begin{equation}\label{cor:CullisAsSumDet:eq}
\det_{n\,k}(X) = \sum_{c \in \binom{[n]}{k}}\sgn_{[n]}(c) \det_{k} \Bigl(X[c|)\Bigr).
\end{equation}
\end{lemma}
\begin{proof}
Indeed, a term-by-term comparison yields
\[
(-1)^{(1 + \ldots + k)} \cdot  (-1)^{c_1 + \ldots + c_k} \begin{vmatrix}x_{c_1\, 1} & \cdots & x_{c_1\, k}\\ \vdots & \ddots & \vdots\\ x_{c_k\, 1} & \cdots & x_{c_k\, k}\end{vmatrix} = \sgn_{[n]}(c) \det_{k} \Bigl(X[c|)\Bigr)
\]
for all $c = \{c_1, \ldots, c_k\} \in c \in \binom{[n]}{k}$, where $1 \le c_1 < \ldots < c_k \le n$.
\end{proof}

Now we list the properties of $\det_{n\, k}$ which are similar to corresponding properties of the ordinary determinant (see \cite[\textsection 5, \textsection 27, \textsection 32]{cullis1913} or~\cite{NAKAGAMI2007422} for detailed proofs).

\begin{theorem}[{\cite[Theorem 13, Theorem 16]{NAKAGAMI2007422}}]\label{thm:DetNKBasProp}
\noindent \begin{enumerate}
\item For $X \in \M_{n}(\mathbb F),$ $\det_{n\,n}(X) = \det (X).$
\item For $X \in \M_{n\,k}(\mathbb F),$ $\det_{n\,k}(X)$ is a linear function of columns of $X$.
\item If a matrix $X \in \M_{n\,k}(\mathbb F)$ has two identical columns or one of its columns is a linear combination of other columns, then $\det_{n\,k}(X)$ is equal to zero.
\item For $X \in \M_{n\,k}(\mathbb F),$ interchanging any two columns of $X$ changes the sign of $\det_{n\,k}(X)$.
\item\label{thm:DetNKBasProp:mark1} Adding a linear combination of columns of $X$ to another column of $X$ does not change $\det_{n\,k}(X)$.
\item For $X \in \M_{n\,k}(\mathbb F),$ $\det_{n\,k}(X)$ can be calculated using the Laplace expansion along a column of $X$ (see Lemma~\ref{lem:DetNKLaplaceExp} for precise formulation).
\end{enumerate}
\end{theorem}

\begin{corollary}\label{cor:CullisBinomialExpansion}Let $n \ge k$, $A, B \in \M_{n\,k} (\F).$% Consider $P \in \F[\lambda]$ defined by an equality $P(\lambda) = \det_{n\, k} (A + \lambda B)$.
Then
\begin{multline}\label{eq:CullisBinomialExpansion}
\det_{n\,k} (A + \lambda B)\\
= \sum_{d = 0}^{k}\lambda^d \left( \sum_{1 \le i_1 < \ldots < i_d \le k} \det_{n\,k}\Bigl(A(|1]\Big|\ldots \Big| B(|i_1] \Big| \ldots \Big| B(|i_d] \Big| \ldots \Big| A(|k] \Bigr)\right),
\end{multline}
where both sides of the equality are considered as formal polynomials in $\lambda$, i.e. as elements of $\F[\lambda]$.
\end{corollary}
\begin{proof}This is a direct consequence from the multilinearity of $\det_{n\,k}$ with respect to the columns of a matrix.
\end{proof}
\begin{corollary}\label{cor:DegDetABLEQK}If $A, B \in \M_{n\,k}(\F)$, then $\deg_\lambda\left(\det_{n\,k} (A + \lambda B)\right) \le k$.
\end{corollary}

%\begin{definition}Suppose that $c \in \binom{[n]}{k}$ and $c = \{i_1, \ldots, i_k\}.$ By $P_c \in \M_{n\,n}(\mathbb F)$ we denote a matrix defined by
%\[
%P_c = E_{i_1\,i_1} + \ldots + E_{i_k\,i_k}.
%\]
%\end{definition}
%
%

%
%\begin{lemma}[{Cf.~\cite[Lemma~8]{NAKAGAMI2007422}}]%\label{lem:DetNKAsSumProj}
%Let $X \in \M_{n\,k}(\mathbb F).$ Then
%$$\sgn_{[n]}(c) \det_{n\,k}(P_cX) = \det_{k} \Bigl(X[c|)\Bigr)$$
%and
%$$\det_{n\,k}(X) = \sum_{c \in \binom{[n]}{k}} \det_{n\,k} (P_c X).$$
%\end{lemma}

\begin{lemma}[Cf.~{\cite[Lemma~20]{NAKAGAMI2007422}}]~\label{lem:NAKAGAMILEMMA20} Assume that $n > k \ge 1$. Let  $X \in \M_{n\,k}(\F)$ and $Y \in \M_{n\,(k+1)}$ is defined by $Y = X | \begin{psmallmatrix}1 \\ \vdots \\ 1\end{psmallmatrix}$. Then
\[
\det_{n\,(k+1)}(Y) = \begin{cases}
\det_{n\,k}(X), & \mbox{$n + k$ is odd},\\
0, & \mbox{$n + k$ is even}.
\end{cases}
\]
\end{lemma}

\begin{lemma}[{Cf.~\cite[Theorem~16]{NAKAGAMI2007422}}]\label{lem:DetNKLaplaceExp}
Let $1 < k \le n$. For any $n \times k$ matrix $X = (x_{i\, j})$ the expansion of $\det_{n\, k}(X)$ along the $j$-th column
is given by
\[
\det_{n\, k}(X) = \sum_{i = 1}^n (-1)^{i+j} x_{i\, j} \det_{(n-1)\,(k-1)}\Bigl(X(i|j)\Bigr),
\]
where $X(i|j)$ denotes the $(n-1) \times (k-1)$ matrix obtained from $Y$ by deleting the $i$-th row and the $j$-th column.
\end{lemma}

\begin{lemma}[{Invariance of $\det_{n\,k}$ under cyclic shifts, Cf.~\cite[Theorem 3.5]{amiri2010}}]\label{lem:PermKNOddCyclic}If $n \ge k$, and $k + n$ is odd, then for all $X = (x_{i\,j}) \in \M_{n\,k}(\F)$ and $i \in \{1, \ldots, n\}$ 
\[
(-1)^{(i+1)k}
\begin{vmatrix}
x_{i\,1} & \ldots & x_{i\,k}\\
\vdots & \ddots & \vdots\\
x_{n\,1} & \ldots & x_{n\,k}\\
x_{1\,1} & \ldots & x_{1\,k}\\
  \vdots & \ddots & \vdots\\
x_{i-1\,1} & \ldots & x_{i-1\,k}\\
\end{vmatrix}_{n\,k} =
\begin{vmatrix}
x_{1\,1} & \ldots & x_{1\,k}\\
  \vdots & \ddots & \vdots\\
x_{n\,1} & \ldots & x_{n\,k}\\
\end{vmatrix}_{n\,k}.
\]
Here the matrix on the left-hand side of the equality is obtained from $X$ by performing the row cyclical shift sending $i$-th row of $X$ to the first row of the result.
\end{lemma}

\begin{lemma}[Invariance of $\det_{n\,k}$ under semi-cyclic shifts, Cf.~{\cite[Theorem 3.6]{amiri2010}}]\label{lem:PermKNEvenSemicyclic}If $n\ge k$, and $k + n$ is even, then for all $X = (x_{i\,j}) \in \M_{n\,k}(\F)$ and $i \in \{1, \ldots, n\}$
\[
(-1)^{(n-i)k}
\begin{vmatrix}
x_{i\,1} & \cdots & x_{i\,k}\\
\vdots & \ddots & \vdots\\
x_{n\,1} & \cdots & x_{n\,k}\\
-x_{1\,1} & \cdots & -x_{1\,k}\\
  \vdots & \ddots & \vdots\\
-x_{(i-1)\,1} & \cdots & -x_{(i-1)\,k}\\
\end{vmatrix}_{n\,k} =
\begin{vmatrix}
x_{1\,1} & \cdots & x_{1\,k}\\
  \vdots & \ddots & \vdots\\
x_{n\,1} & \cdots & x_{n\,k}\\
\end{vmatrix}_{n\,k}.
\]
Here the matrix on the left-hand side of the equality is obtained from $X$ by performing the following sequence of operations: the row cyclical shift sending $i$-th row of $X$ to the first row of the result; multiplying the bottom $i-1$ rows by $-1$.
\end{lemma}

In the following lemma we introduce an invertible linear map satisfying certain properties which will be used in the proof of Lemma~\ref{thm:MaxDimKEvenAltSumZero}.

\begin{lemma}\label{lem:SCS}Let $k$ be an odd integer, $n \ge k \ge 1$ and $1 \le i^{\circ} \le n$. Then there exists an invertible linear map $\SCS_{i^{\circ}} \colon \M_{n\, k}(\F) \to \M_{n\,k}(\F)$ having the following properties:
\begin{enumerate}[label=(S\arabic*), ref=(S\arabic*)]
%\item\label{SCSprop1} $\left(\SCS_{i^{\circ}}(X)\right)[1|) = X[i|)$ for all $X \in \M_{n\,k}(\F)$;
\item\label{SCSprop2} if $X \in \M_{n\,k}(\F)$ and $\mathbf z \in \F^n$ are such that $\mathbf z^t X = 0$, then $\mathbf {z^{\circ}}^t\SCS_{i^{\circ}}(X) = 0$, where 
    \begin{equation}\label{lem:SCS:eq}\mathbf z^{\circ} = (\mathbf z_{i^{\circ}}, \ldots, \mathbf z_n, (-1)^{n+k+1}\mathbf z_1 \ldots, \mathbf (-1)^{n+k+1}\mathbf z_{i^{\circ}-1})^t;
    \end{equation}
\item\label{SCSprop3} $\det_{n\,k}(X) = 0 \Rightarrow \det_{n\,k}(\SCS_{i^\circ}(X)) = 0$;
\item\label{SCSprop4} If $k$ is odd, then $$\sum_{i=1}^n (-1)^{i}\left(\SCS_{i^\circ}(X)\right)[i|) = \begin{pmatrix}0 & \cdots & 0\end{pmatrix}$$ implies that $$\sum_{i=1}^n (-1)^{i}X[i|) = \begin{pmatrix}0 & \cdots & 0\end{pmatrix}$$ for all $X \in \M_{n\,k}(\F)$.
\end{enumerate} 
\end{lemma}
\begin{proof}
The proof of this lemma relies on Lemma~\ref{lem:PermKNOddCyclic} and Lemma~\ref{lem:PermKNEvenSemicyclic}, which are applicable for the case of $n + k$ odd and $n + k$ is even, respectively. For this reason, each of these two cases will be considered separately.
\paragraph{The case if $n + k$ is odd.} Let  $\SCS_{i^\circ} \colon \M_{n\, k}(\F) \to \M_{n\,k}(\F)$ be defined by
\begin{equation}\label{lem:SCS:eq1}
\SCS_{i^\circ}(X) = \begin{pmatrix}
x_{i^{\circ}\,1} & \cdots & x_{i^{\circ}\,k}\\
\vdots & \ddots & \vdots\\
x_{n\,1} & \ldots & x_{n\,k}\\
\vdots & \ddots & \vdots\\
x_{(i^{\circ}-1)\,1} & \ldots & x_{(i^{\circ}-1)\,k}\\
\end{pmatrix}\;\;\mbox{for all}\;\; X \in \M_{n\,k}(\F).
\end{equation}

This map is clearly linear and invertible. Now we verify that $\SCS_{i^\circ}$ indeed satisfies the properties~\ref{SCSprop2}--\ref{SCSprop4}.

\paragraph{Property~\ref{SCSprop2}.} Let $X \in \M_{n\,k}(\F)$ and $\mathbf z \in \F^n$ be such that $\mathbf z^t X = 0$. It follows from~\eqref{lem:SCS:eq1} that $\mathbf{z^{\mathrm{odd}}}\SCS_{i^{\circ}}(X) = 0$, where $\mathbf{z^{\mathrm{odd}}}$ is defined by
$\mathbf{z^{\mathrm{odd}}} = (\mathbf z_{i^{\circ}}, \ldots, \mathbf z_n, \mathbf z_1, \ldots, \mathbf z_{i^{\circ}-1})^t.$ Since $n + k$ is odd, 
\begin{equation*}
\mathbf{z^{\mathrm{odd}}} = (\mathbf z_{i^{\circ}}, \ldots, \mathbf z_n, \mathbf z_1, \ldots, \mathbf z_{i^{\circ}-1})^t\\
= (\mathbf z_{i^{\circ}}, \ldots, \mathbf z_n, (-1)^{n + k + 1}\mathbf z_1, \ldots, (-1)^{n + k + 1}\mathbf z_{i^{\circ}-1})^t = \mathbf z^{\circ},
\end{equation*}
where $\mathbf z^{\circ}$ be as defined in~\eqref{lem:SCS:eq}.

\paragraph{Property~\ref{SCSprop3}.} Let $X \in \M_{n\,k}(\F)$ be such that $\det_{n\,k}(X)$ = 0. Then by Lemma~\ref{lem:PermKNOddCyclic} we have
\[
\det_{n\,k}(\SCS_{i^\circ}(X)) = (-1)^{(i^{\circ}+1)k}\det_{n\,k}(X) = 0.
\]

\paragraph{Property~\ref{SCSprop4}.} Assume that $k$ is odd. Let $X \in \M_{n\,k}$ be such that 
\begin{equation}\label{lem:SCS:eq21}
\sum_{i=1}^n (-1)^{i}X^{\circ}[i|) = \begin{pmatrix}0 & \cdots & 0\end{pmatrix},
\end{equation}
 where $X^{\circ} = \SCS_{i^{\circ}}(X)$. Rewrite the sum $\sum_{i=1}^{n} (-1)^{i} X[i|)$ as follows
\begin{equation*}
\sum_{i=1}^{n} (-1)^{i} X[i|) = \sum_{i=1}^{i^{\circ}-1} (-1)^{i} X[i|) + \sum_{i=i^{\circ}}^n (-1)^{i} X[i|).
\end{equation*}
Since $n$ is even in this case,
\[
\sum_{i=1}^{i^{\circ}-1} (-1)^{i} X[i|) = \sum_{i=n-i^{\circ}+2}^{n} (-1)^{i - n + i^{\circ} - 1} X^{\circ}[i|)\quad\mbox{and}\quad\sum_{i=i^{\circ}}^n (-1)^{i} X[i|) = \sum_{i=1}^{n-i^{\circ}+1} (-1)^{i + i^{\circ} - 1} X^{\circ}[i|),
\]
then
\begin{multline*}
\sum_{i=1}^{n} (-1)^{i} X[i|) = \sum_{i=1}^{i^{\circ}-1} (-1)^{i} X[i|) + \sum_{i=i^{\circ}}^n (-1)^{i} X[i|)\\
= \sum_{i=n-i^{\circ}+2}^{n} (-1)^{i - n + i^{\circ} - 1} X^{\circ}[i|) + \sum_{i=1}^{n-i^{\circ}+1} (-1)^{i + i^{\circ} - 1} X^{\circ}[i|)\phantom{XXXXXXXXXXX}\\
= \sum_{i=n-i^{\circ}+2}^{n} (-1)^{i + i^{\circ} - 1} X^{\circ}[i|) + \sum_{i=1}^{n-i^{\circ}+1} (-1)^{i + i^{\circ} - 1} X^{\circ}[i|) = (-1)^{i^{\circ} - 1} \sum_{i=1}^n (-1)^{i} X^{\circ}[i|).
\end{multline*}
Thus, using the equality~\eqref{lem:SCS:eq21} we conclude that
\[
\sum_{i=1}^{n} (-1)^{i} X[i|) = (-1)^{i^{\circ} - 1} \sum_{i=1}^{n} (-1)^{i} X^{\circ}[i|) = \begin{pmatrix}0 & \cdots & 0\end{pmatrix}.
\] 

\paragraph{The case if $n + k$ is even.} Let us define $\SCS_{i^{\circ}} \colon \M_{n\, k}(\F) \to \M_{n\,k}(\F)$ by
\[
\SCS_{i^{\circ}}(X) = \begin{pmatrix}
x_{i^{\circ}\,1} & \ldots & x_{i^{\circ}\,k}\\
\vdots & \vdots & \vdots\\
x_{n\,1} & \ldots & x_{n\,k}\\
-x_{1\,1} & \ldots & -x_{1\,k}\\
  \vdots & \vdots & \vdots\\
-x_{(i^{\circ}-1)\,1} & \ldots & -x_{(i^{\circ}-1)\,k}\\
\end{pmatrix}\;\;\mbox{for all}\;\;X \in \M_{n\,k}(\F).
\]
This map is clearly linear and invertible. Now we verify that $\SCS_{i^\circ}$ indeed satisfies the properties~\ref{SCSprop2}--\ref{SCSprop4}.

\paragraph{Property~\ref{SCSprop2}.} Let $X \in \M_{n\,k}(\F)$ and $\mathbf z \in \F^n$ be such that $\mathbf z^t X = 0$. It follows from~\eqref{lem:SCS:eq1} that $\mathbf{z^{\mathrm{odd}}}\SCS_{i^{\circ}}(X) = 0$, where $\mathbf{z^{\mathrm{even}}}$ is defined by
$\mathbf{z^{\mathrm{even}}} = (\mathbf z_{i^{\circ}}, \ldots, \mathbf z_n, -\mathbf z_1, \ldots, -\mathbf z_{i^{\circ}-1})^t.$ Since $n + k$ is even, 
\begin{multline*}
\mathbf{z^{\mathrm{even}}} = (\mathbf z_{i^{\circ}}, \ldots, \mathbf z_n, -\mathbf z_1, \ldots, -\mathbf z_{i^{\circ}-1})^t\\
= (\mathbf z_{i^{\circ}}, \ldots, \mathbf z_n, (-1)^{n + k + 1}\mathbf z_1, \ldots, (-1)^{n + k + 1}\mathbf z_{i^{\circ}-1})^t = \mathbf z^{\circ},
\end{multline*}
where $\mathbf z^{\circ}$ be as defined in~\eqref{lem:SCS:eq}.
\paragraph{Property~\ref{SCSprop3}.}   Let $X \in \M_{n\,k}(\F)$ be such that $\det_{n\,k}(X)$ = 0. Then by Lemma~\ref{lem:PermKNEvenSemicyclic}  we have
\[
\det_{n\,k}(\SCS_{i^\circ}(X)) = (-1)^{(n-i^{\circ})k}\det_{n\,k}(X) = 0.
\]

\paragraph{Property~\ref{SCSprop4}.} Assume that $k$ is odd. Let $X \in \M_{n\,k}$ be such that 
\begin{equation}\label{lem:SCS:eq22}
\sum_{i=1}^n (-1)^{i}X^{\circ}[i|) = \begin{pmatrix}0 & \cdots & 0\end{pmatrix},
\end{equation}
 where $X^{\circ} = \SCS_{i^{\circ}}(X)$. Rewrite the sum $\sum_{i=1}^{n} (-1)^{i} X[i|)$ as follows 
\begin{equation*}
\sum_{i=1}^n (-1)^{i} X[i|) = \sum_{i=1}^{i^{\circ}-1} (-1)^{i} X[i|) + \sum_{i=i^{\circ}}^n (-1)^{i} X[i|).
\end{equation*}
Since $n$ is odd in this case,
\[
\sum_{i=1}^{i^{\circ}-1} (-1)^{i} X[i|) = -\left(\sum_{i=n-i^{\circ}+2}^{n} (-1)^{i - n + i^{\circ} - 1} X^{\circ}[i|)\right)
\]
and
\[\sum_{i=i^{\circ}}^n (-1)^{i} X[i|) = \sum_{i=1}^{n-i^{\circ}+1} (-1)^{i + i^{\circ} - 1} X^{\circ}[i|),
\]
then
\begin{multline*}
\sum_{i=1}^n (-1)^{i} X[i|) = \sum_{i=1}^{i^{\circ}-1} (-1)^{i} X[i|) + \sum_{i=i^{\circ}}^n (-1)^{i} X[i|)\\
= -\left(\sum_{i=n-i^{\circ}+2}^{n} (-1)^{i - n + i^{\circ} - 1} X^{\circ}[i|)\right) + \sum_{i=1}^{n-i^{\circ}+1} (-1)^{i + i^{\circ} - 1} X^{\circ}[i|)\phantom{XXXXXXXXX}\\
= \sum_{i=n-i^{\circ}+2}^{n} (-1)^{i + i^{\circ} - 1} X^{\circ}[i|) + \sum_{i=1}^{n-i^{\circ}+1} (-1)^{i + i^{\circ} - 1} X^{\circ}[i|) = (-1)^{i^{\circ} - 1} \sum_{i=1}^n (-1)^{i} X^{\circ}[i|).
\end{multline*}
Thus, using the equality~\eqref{lem:SCS:eq22} we conclude that
\[
\sum_{i=1}^n (-1)^{i} X[i|) = (-1)^{i^{\circ} - 1} \sum_{i=1}^n (-1)^{i} X^{\circ}[i|) = \begin{pmatrix}0 & \cdots & 0\end{pmatrix}.
\]
\end{proof}

\section{Linear varieties of matrices annihilating $\det_{n\,k}$}
\label{sec:linvarsann}

In this section we prove our main result regarding linear varieties in matrix spaces annihilating $\det_{n\,k}$ (Theorem~\ref{thm:MaxDimKEvenAltSumZero}). 

Note that our argumentation in the proofs of Lemma~\ref{lem:ALSDetNKZeroDimKThenExistVec} and Theorem~\ref{thm:MaxDimKEvenAltSumZero} is motivated by the argumentation used in the proof of~\cite[Theorem~2]{Dieudonne1948} and similar to it, but for convenience we use the language of matroid theory. % convienient . % to make the argumentation clear.

First we define the projection map which formalizes the concept of operation of striking out one row and one column of matrix. For example, this operation is performed when we use the Laplace expansion formula.

\begin{definition}Let $n \ge k > 1$ and $1 \le i \le n, 1 \le j \le k$. We denote by\linebreak $\ins_{n\, k}^{(i,j)} \colon [n-1]\times [k-1] \to [n]\times [k]$ % \setminus \left(\{i\} \times [k] \cup [n] \times \{j\}\right)$
a map defined by
\[
\ins_{n\, k}^{(i,j)}((a,b)) = \begin{cases}
(a+1,b+1), & a \ge i, b \ge j\\
(a+1,b), & a \ge i, b < j\\
(a,b+1), & a < i, b \ge j\\
(a,b), & \mbox{otherwise}
\end{cases}
\]
for all $(a,b) \in [n-1]\times [k-1]$.
\end{definition}

\begin{lemma}\label{lem:InsPresRowCol}Let $n, k \ge 1$, $1 \le i \le n, 1 \le j \le k$. The pairs $(a,b), (a',b') \in [n-1] \times [k-1]$ belong to the same row (column) if and only if  $\ins_{n\, k}^{(i,j)}((a,b)), \ins_{n\, k}^{(i,j)}((a',b')) \in [n] \times [k]$ belong to the same row (column).
\end{lemma}
\begin{proof}
It follows from the definition of $\ins_{n\, k}^{(i,j)}$ that $a = a'$ if and only if the first elements of tuples $\ins_{n\, k}^{(i,j)}((a,b))$ and $\ins_{n\, k}^{(i,j)}((a',b'))$ are equal.

Similarly,  $b = b'$ if and only if the second elements of tuples $\ins_{n\, k}^{(i,j)}((a,b))$ and\linebreak $\ins_{n\, k}^{(i,j)}((a',b'))$ are equal.
\end{proof}

In the following lemma we provide an algorithm to obtain a linear variety in\linebreak $\M_{(n-1)\,(k-1)}(\F)$ annihilating $\det_{(n-1)\,(k-1)}$ from a given linear variety in $\M_{n\,k}(\F)$ annihilating $\det_{n\,k}$ and a relationship between their codimensions and cobases.

\begin{lemma}\label{lem:StrikingOutLifting}Let $n \ge k > 1$, $1 \le i' \le n$, $1 \le j' \le k$, $\K \subseteq \M_{n\,k}(\F)$ be a linear variety, $B^* \in \mathcal B(\m(\K)^*)$ and $c_1, \ldots, c_n \in \F$.  Assume that 
\begin{equation}\label{lem:StrikingOutLifting:eqq1}
B^* \cap \left([n]\times \{j'\}\right) = \varnothing.
\end{equation}
Then
\begin{enumerate}[label=(\alph*), ref=(\alph*)]
\item\label{lem:StrikingOutLifting:part4}$\K' = \pi_{\ins_{n\,k}^{(i',j')}}(\K \cap U(x_{(i,j')} = c_i \mid 1 \le i \le n))$ is a linear variety.
\item\label{lem:StrikingOutLifting:part5} 
\begin{equation}\label{lem:StrikingOutLifting:ineq}
\codim(\K') \le \codim(\K) - \Big|B^* \cap \big(\{i'\} \times [k]\big)\Big| .
\end{equation}
\item\label{lem:StrikingOutLifting:part1} If ${B'}^*$ is a cobasis of $\m(\K')$, then 
\begin{equation}\label{lem:StrikingOutLifting:eqq2}
\big((B^* \cap \left(\{i'\}\times [k]\right)) \cup \ins_{n\,k}^{(i',j')}({B'}^*)\big) \in \mathcal I(\m^*(\K)).
\end{equation}
\item\label{lem:StrikingOutLifting:part3} If 
\begin{equation}\label{lem:StrikingOutLifting:eqq3}
\det_{n\,k} (X) = 0\;\;\mbox{for all}\;\;X \in \K
\end{equation}
 and $\{c_i = \delta_{i\,i'} \mid 1 \le i \le n\}$, then 
 \begin{equation}\label{lem:StrikingOutLifting:eqq4}
 \det_{(n-1)\,(k-1)} (X') = 0\;\;\mbox{for all}\;\;X' \in \K'.
 \end{equation}
\end{enumerate}
\end{lemma} 

\begin{proof}First note that $I = [n]\times \{j'\}$ is an independent set of $\m(\K)$ because $B^*$ is a cobasis of $\m(\K)$ and the equality~\eqref{lem:StrikingOutLifting:eqq1} holds. This fact is used throughout the proof.

\paragraph{\ref{lem:StrikingOutLifting:part4}} By applying Lemma~\ref{lem:LinVarSec} to $\K$, $I$ and $\{c_i\}$ we obtain that $\K \cap U(x_{(i,j')} = c_i \mid 1 \le i \le n)$.  Therefore, $\K'$ is a linear variety by Lemma~\ref{lem:LinImageAS} because $\pi_{\ins_{n\,k}^{(i',j')}}$ is a linear map.

\paragraph{\ref{lem:StrikingOutLifting:part5} and \ref{lem:StrikingOutLifting:part1}} From Lemma~\ref{lem:LinVarSecProj} applied to $\K$, $I = [n]\times \{j'\}$  and $f = \ins_{n\,k}^{(i',j')}$ we conclude that 
\begin{multline}\label{lem:StrikingOutLifting:eq10}
\ins_{n\,k}^{(i',j')}\;\mbox{is an isomorphism between}\;\m(\K')\;\mbox{and}\;\m',\;\\
\mbox{where}\;\m' = \biggl(\m(\K)/\left([n]\times \{j'\}\right)\biggr)|\biggl(\ins_{n\,k}^{(i',j')}\left([n-1]\times [k-1]\right)\biggr)
\end{multline} 
because $f^{-1}(I) = \left(\ins_{n\,k}^{(i',j')}\right)^{-1}\left([n]\times \{j'\}\right) = \varnothing$ by~\eqref{lem:StrikingOutLifting:eqq1}. This claim is used in the proofs of the parts~\ref{lem:StrikingOutLifting:part5} and \ref{lem:StrikingOutLifting:part1} of the lemma.\bigskip

\noindent Let us prove the part~\ref{lem:StrikingOutLifting:part5} of the lemma. By the claim~\eqref{lem:StrikingOutLifting:eq10} we have
\begin{equation}\label{lem:DetNKZeroStrikingOutDetZero:eq102}
\codim(\K') = r^*\left(\m(\K')\right) = r^*\left(\m'\right).
\end{equation}
Corollary~\ref{lem:MatroidCobaseContractionRestriction}\ref{lem:MatroidCobaseContractionRestriction:part2} applied to $\m = \m(\K),$ $E = [n]\times [k]$, $T =[n]\times \{j'\}$,\linebreak $S = \ins_{n\,k}^{(i',j')}\left([n-1]\times [k-1]\right)$ and $I^* = B^*$ implies that
\begin{equation}\label{lem:DetNKZeroStrikingOutDetZero:eq2}
r^*\left(\m'\right)
\le r^*\left(\m(\K)\right) - \Bigg|\bigg(B^* \setminus \left([n]\times \{j'\}\right)\bigg) \setminus \left(\ins_{n\,k}^{(i',j')}\left([n-1]\times [k-1]\right)\right)\Bigg|
\end{equation}

Let us simplify the left-hand side of~\eqref{lem:DetNKZeroStrikingOutDetZero:eq2}. First,
\begin{equation}\label{lem:DetNKZeroStrikingOutDetZero:eq5}
r^*\left(\m(\K)\right) = \codim(\K)
\end{equation}
by Corollary~\ref{cor:matLinVarDim}. Second, the equality~\eqref{lem:StrikingOutLifting:eqq1} implies that
\begin{equation}\label{lem:DetNKZeroStrikingOutDetZero:eq6}
\Big(B^* \setminus \left([n]\times \{j'\}\right)\Big) \setminus \left(\ins_{n\,k}^{(i',j')}\big([n-1]\times [k-1]\big)\right) = B^* \setminus \left(\ins_{n\,k}^{(i',j')}\big([n-1]\times [k-1]\big)\right).
\end{equation}
Since 
$\ins_{n\,k}^{(i',j')}\big([n-1]\times [k-1]\big) = \Big([n]\times [k]\Big) \setminus \Big(\big([n]\times \{j'\}\big) \cup \big(\{i'\}\times [k]\big)\Big)$
and $B^* \subseteq \big([n]\times [k]\big)$, then
\begin{equation}\label{lem:DetNKZeroStrikingOutDetZero:eq7}
B^* \setminus \left(\ins_{n\,k}^{(i',j')}\big([n-1]\times [k-1]\big)\right)\\
= B^* \cap \Big(\big([n]\times \{j'\}\big) \cup \big(\{i'\}\times [k]\big)\Big).
\end{equation}
By the equality~\eqref{lem:StrikingOutLifting:eqq1} we have
\begin{equation}\label{lem:DetNKZeroStrikingOutDetZero:eq8}
B^* \cap \Big(\big([n]\times \{j'\}\big) \cup \big(\{i'\}\times [k]\big)\Big) = B^* \cap \big(\{i'\}\times [k]\big).
\end{equation}
The equalities~\eqref{lem:DetNKZeroStrikingOutDetZero:eq6}--\eqref{lem:DetNKZeroStrikingOutDetZero:eq8} aligned together yield the equality
\begin{equation}\label{lem:DetNKZeroStrikingOutDetZero:eq4}
\Big(B^* \setminus \big([n]\times \{j'\}\big)\Big) \setminus \Big(\ins_{n\,k}^{(i',j')}\big([n-1]\times [k-1]\big)\Big) = B^* \cap \big(\{i'\}\times [k]\big).
\end{equation}
Therefore, by substituting~\eqref{lem:DetNKZeroStrikingOutDetZero:eq102}, \eqref{lem:DetNKZeroStrikingOutDetZero:eq5} and~\eqref{lem:DetNKZeroStrikingOutDetZero:eq4} into the inequality~\eqref{lem:DetNKZeroStrikingOutDetZero:eq2} we obtain the inequality~\eqref{lem:StrikingOutLifting:ineq}. This establishes the part~\ref{lem:StrikingOutLifting:part5} of the lemma.\bigskip

\noindent Let us prove the part~\ref{lem:StrikingOutLifting:part1} of the lemma. Let ${B'}^*$ be a cobasis of $\m(\K')$ and ${I'}^*$ be defined by
\begin{equation}\label{lem:DetNKZeroStrikingOutDetZero:eq9}
{I'}^* = \bigg(\Big(B^* \setminus \big([n]\times \{j'\}\big)\Big) \setminus \Big(\ins_{n\,k}^{(i',j')}\big([n-1]\times [k-1]\big)\Big)\bigg) \cup \ins_{n\,k}^{(i',j')}\left({B'}^*\right).
\end{equation}
 The set $\ins_{n\,k}^{(i',j')}({B'}^*)$ is a cobasis of $\m'$ as it follows from~\eqref{lem:StrikingOutLifting:eq10}. Hence, \begin{equation}\label{lem:DetNKZeroStrikingOutDetZero:eq901}
 {I'}^* \in \mathcal I(\m^*(\K))
  \end{equation}
  by Corollary~\ref{lem:MatroidCobaseContractionRestriction}\ref{lem:MatroidCobaseContractionRestriction:part1} applied to $\m = \m(\K),$ $E = [n]\times [k]$, $T =[n]\times \{j'\}$,  $S = \ins_{n\,k}^{(i',j')}\big([n-1]\times [k-1]\big)$, $I^* = B^*$ and $\ins_{n\,k}^{(i',j')}\left({B'}^*\right)$. By substituting~\eqref{lem:DetNKZeroStrikingOutDetZero:eq4} into~\eqref{lem:DetNKZeroStrikingOutDetZero:eq9} we obtain that 
\begin{equation}\label{lem:DetNKZeroStrikingOutDetZero:eq902}
 {I'}^* = \Big(B^* \cap \big(\{i'\}\times [k]\big)\Big) \cup \ins_{n\,k}^{(i',j')}\left({B'}^*\right).
\end{equation}
Then, aligning~\eqref{lem:DetNKZeroStrikingOutDetZero:eq901} and~\eqref{lem:DetNKZeroStrikingOutDetZero:eq902} together yields the equality~\eqref{lem:StrikingOutLifting:eqq2} as it is claimed in the part~\ref{lem:StrikingOutLifting:part1} of the lemma.

\paragraph{\ref{lem:StrikingOutLifting:part3}} Let  $X' = (x'_{i\,j}) \in \K'$. It means that there exists $X = (x_{i\,j}) \in A \cap U(x_{(i,j')} = \delta_{i\,i'} \mid 1 \le i \le n)$ such that $X' = \pi_{\ins_{n\,k}^{(i',j')}}(X) = X(i'|j')$. This implies that $X$ has the form
\begin{equation}\label{lem:DetNKZeroStrikingOutDetZero:eq1}
\begin{pmatrix}
x'_{1\,1} & \cdots & x'_{1\, (j'-1)} & 0 & x'_{1\,j'} & \cdots & x'_{1\, (k-1)}\\
\vdots & \ddots & \vdots & \vdots & \vdots & \ddots & \vdots\\
x'_{(i'-1)\,1} & \cdots & x'_{(i'-1)\, (j'-1)} & 0 & x'_{(i'-1)\,j'} & \cdots & x'_{(i'-1)\,(k-1)}\\
* & \cdots & * & 1 & * & \cdots & *\\
x'_{i'\,1} & \cdots & x'_{i'\, (j'-1)} & 0 & x'_{i'\,j'} & \cdots & x'_{i'\, (k-1)}\\
\vdots & \ddots & \vdots & \vdots & \vdots & \ddots & \vdots\\
x'_{(n-1)\,1} & \cdots & x'_{(n-1)\, (j'-1)} & 0 & x'_{(n-1)\,j'} & \cdots & x'_{(n-1)\,(k-1)}
\end{pmatrix}.
\end{equation}

By the Laplace expansion along the $j'$-th column (Lemma~\ref{lem:DetNKLaplaceExp}), we have an equality
\begin{equation}\label{lem:DetNKZeroStrikingOutDetZero:eq10001}
\det_{n\, k}(X) = \sum_{i = 1}^n (-1)^{i+j'} x_{i\, j'} \det_{(n-1)\,(k-1)}\Bigl(X(i|j)\Bigr),
\end{equation}
where $X(i|j)$ denotes a matrix obtained from $X$ by striking out the $i$-th row and the $j$-th column. Since $x_{i'\,j'} = 1$ is the only nonzero value in the sequence $x_{1\,j'}, \ldots, x_{n\,j'}$, then
\begin{equation}\label{lem:DetNKZeroStrikingOutDetZero:eq10002}
\sum_{i = 1}^n (-1)^{i+j'} x_{i\, j'} \det_{(n-1)\,(k-1)}\Bigl(X(i|j)\Bigr) = (-1)^{i'+j'}\det_{(n-1)\,(k-1)}\Bigl(X(i'|j')\Bigr).
\end{equation}
Note that $X(i'|j') = X'$ because $X$ has the form~\eqref{lem:DetNKZeroStrikingOutDetZero:eq1}. Therefore, 
\begin{equation}\label{lem:DetNKZeroStrikingOutDetZero:eq10003}
\det_{(n-1)\,(k-1)}(X') = \det_{(n-1)\,(k-1)}\Bigl(X(i'|j')\Bigr)\\ \overset{\eqref{lem:DetNKZeroStrikingOutDetZero:eq10001}\;\mbox{\scriptsize and}\;\eqref{lem:DetNKZeroStrikingOutDetZero:eq10002}}{=\joinrel=\joinrel=\joinrel=\joinrel=\joinrel=\joinrel=} (-1)^{i' + j'}\det_{n\, k}(X) \overset{\eqref{lem:StrikingOutLifting:eqq3}}{=\joinrel=\joinrel=} 0.
\end{equation}
Since~\eqref{lem:DetNKZeroStrikingOutDetZero:eq10003} holds for all $X' \in \K'$, the equality~\eqref{lem:StrikingOutLifting:eqq4} is established.
\end{proof}

\begin{corollary}\label{cor:LetASLcodimStrikingOutIfColumnPrecise}Let $k > 1$, $\K \subseteq \M_{n\,k}(\mathbb F)$ be a linear variety such that $\det_{n\,k} (X) = 0$ for all $X \in A$, and $B^* \in \mathcal B(\m(\K)^*)$. If there exist $1 \le i' \le n$ and $1 \le j' \le k$ such that $\Big|B^* \cap \big(\{i'\}\times [k]\big)\Big| = l$ and $B^* \cap \big([n]\times \{j'\}\big) = \varnothing$, then there exists a linear variety $\K' \subseteq \M_{(n-1)\,(k-1)}(\mathbb F)$ with $\codim(\K') \le \codim(\K) - l$ such that $\det_{(n-1)\,(k-1)} (X') = 0$ for all $X' \in \K'$.
\end{corollary}

\begin{corollary}\label{cor:LetASLcodimStrikingOutIfColumnAtLeastOne} Let $k > 1$, $\K \subseteq \M_{n\,k}(\mathbb F)$ be a linear variety such that $\det_{n\,k} (X) = 0$ for all $X \in A$, and  $B^* \in \mathcal B(\m(\K)^*)$. If there exists $1 \le j' \le k$ such that $B^* \cap \big([n]\times \{j'\}\big) = \varnothing$, then there exists a linear variety $\K' \subseteq \M_{(n-1)\,(k-1)}(\mathbb F)$ with $\codim(\K') \le \codim(\K) - 1$ such that $\det_{(n-1)\,(k-1)} (X') = 0$ for all $X' \in \K'$.
\end{corollary}

Corollary~\ref{cor:LetASLcodimStrikingOutIfColumnAtLeastOne} allows us to provide a lower bound the codimension of linear variety consisting of matrices annihilating $\det_{n\,k}$.

\begin{theorem}\label{thm:ALSDetNKZeroCodimGEK}Let $n \ge k \ge 1$ and $\K \subseteq \M_{n\,k}(\mathbb F)$ be a linear variety such that $\det_{n\,k} (X) = 0$ for every $X \in \K.$ Then $\codim (\K) \ge k.$
\end{theorem}
\begin{proof}
We prove it by induction on $k$. If $k = 1$ and $\K$ is such that $\det_{n\,1} (X) = 0$ for every $X \in \K$, then this condition implies that $\sum_{i=1}^{n}(-1)^{i-1}x_{i\,1} = 0$. Hence, 
$$\K \subseteq \AS\big((1, -1, 1, \ldots, (-1)^{n-1}),0\big)$$
 and consequently
\[
\codim (\K) \ge \codim\Big(\AS\big((1, -1, 1, \ldots, (-1)^{n-1}),0\big)\Big) = 1 = k.
\]

Assume now that $k \ge 2$, the statement of the lemma holds for $k-1$ and $\K \subseteq \M_{n\,k}(\mathbb F)$ is a linear variety such that $\det_{n\, k} (X) = 0$ for every $X \in \K$. We will prove that $\codim (\K) \ge k$ by contradiction. Suppose that $\codim (\K) < k.$ Let $B^*$ be any cobasis of $\m(\K)$. Since $\left|B^*\right| = \codim (\K) < k$, there exists $1 \le j' \le k$ such that $B^* \cap \big([n]\times \{j'\}\big) = \varnothing$. Hence, $B^*, \K$ and $j'$ satisfy the conditions of Corollary~\ref{cor:LetASLcodimStrikingOutIfColumnAtLeastOne} and there exists a linear variety $\K' \subseteq \M_{n-1\,k-1}(\F)$ with
\[\codim(\K') \le \codim(\K) - 1 < k - 1\]
which leads to a contradiction. Therefore, $\codim (\K) \ge k$.
\end{proof}

Having the linear variety $\K$ annihilating $\det_{n\,k}$ of minimal possible codimension, by applying Corollary~\ref{cor:LetASLcodimStrikingOutIfColumnPrecise} we obtain a restrictive condition on the cobases of $\m(\K)$, as detailed in the lemma below. This condition will be used in the sequel. %does not contain any elements of some column, then it does not contain two elements belonging to the same row. % $[n]\times [k]$.

\begin{lemma}\label{lem:ALSDetNKZeroCapLE1}Assume that  $n \ge k > 1$. Let $\K \subseteq \M_{n\,k}(\mathbb F)$ be a linear variety such that $\codim(\K) = k$ and $\det_{n\,k} (X) = 0$ for all $X \in \K$, and $B^* \in \mathcal B(\m(\K)^*)$. If there is  $1 \le j' \le k$ such that $B^* \cap \big([n]\times \{j'\}\big) = \varnothing$, then 
$$\Big|B^* \cap \big(\{i\}\times [k]\big)\Big| \le 1\;\;\mbox{for all}\;\; 1 \le i \le n. $$
\end{lemma}
\begin{proof}By contradiction. Suppose that $\K$ and $B^*$ satisfy the conditions of the lemma and $1 \le i' \le n$ is such that
\begin{equation}\label{lem:ALSDetNKZeroCapLE1:eq1}
\Big|B^* \cap \big(\{i'\}\times [k]\big)\Big| = l > 1.
\end{equation}
Hence, by Corollary~\ref{cor:LetASLcodimStrikingOutIfColumnPrecise} there exists a linear variety $\K' \subseteq \M_{(n-1)\,(k-1)}(\mathbb F)$ with
\[
\codim(\K') \le \codim(\K) - l \overset{\eqref{lem:ALSDetNKZeroCapLE1:eq1}}{=\joinrel=\joinrel=} k - l < k - 1
\]
such that $\det_{(n-1)\,(k-1)} (X') = 0$ for all $X' \in \K'$. This contradicts Theorem~\ref{thm:ALSDetNKZeroCodimGEK}. Thus, the statement of the lemma is true.
\end{proof}

In the following lemma we prove that every linear variety annihilating $\det_{n\,k}$ of minimal possible codimension is a vector space of matrices whose rows satisfy a common linear relation.

\begin{lemma}\label{lem:ALSDetNKZeroDimKThenExistVec}
Assume that $n \ge k + 2$. Let $\K \subseteq \M_{n\,k}(\mathbb F)$ be a linear variety with $\codim(\K) = k$ such that $\det_{n\,k} (X) = 0$ for every $X \in \K.$ Then the following is true:
\begin{enumerate}[label=\textit{\textbf{P\arabic*.}}, ref=\textit{\textbf{P\arabic*}}] %[label=\textsc{P\arabic*.}, ref=\textsc{P\arabic*}]
\item\label{lem:ALSDetNKZeroDimKThenExistVec:P1} If $B^*$ is a cobasis of $\m(\K)$, then $B^* \cap \big([n] \times \{j\}\big) \neq \varnothing$ for all $1 \le j \le k$.
\item\label{lem:ALSDetNKZeroDimKThenExistVec:P2} There exists $\mathbf 0 \neq \mathbf z \in \mathbb F^n$  such that $\K = \{X \in \M_{n\,k}(\F) \mid \mathbf z^t X = \mathbf 0\}.$
\end{enumerate}
\end{lemma}

\begin{proof}The statement of the theorem is clearly true for $k = 1$. To prove it for $k > 1$ we prove separately the following claims:
\begin{enumerate}[label=(\alph*), ref=(\alph*)]
\item\label{lem:ALSDetNKZeroDimKThenExistVec:part1} \ref*{lem:ALSDetNKZeroDimKThenExistVec:P1} holds for $k = 2$.
\item\label{lem:ALSDetNKZeroDimKThenExistVec:part2} If $k \ge 3$, then~\ref*{lem:ALSDetNKZeroDimKThenExistVec:P2} for $k-1$ implies~\ref*{lem:ALSDetNKZeroDimKThenExistVec:P1} for $k$.
\item\label{lem:ALSDetNKZeroDimKThenExistVec:part3} If~\ref*{lem:ALSDetNKZeroDimKThenExistVec:P1} holds for $k$, then~\ref*{lem:ALSDetNKZeroDimKThenExistVec:P2} holds for $k$.
\end{enumerate}

The corresponding diagram of implications obtained from~\ref*{lem:ALSDetNKZeroDimKThenExistVec:part1}, \ref*{lem:ALSDetNKZeroDimKThenExistVec:part2} and \ref*{lem:ALSDetNKZeroDimKThenExistVec:part3} could be represented as follows
\begin{center}
\begin{tikzcd}[arrows=Rightarrow]
\text{\ref*{lem:ALSDetNKZeroDimKThenExistVec:P1}(2) = \ref*{lem:ALSDetNKZeroDimKThenExistVec:part1}}\ar{d}{\text{\ref*{lem:ALSDetNKZeroDimKThenExistVec:part3}}} & \text{\ref*{lem:ALSDetNKZeroDimKThenExistVec:P1}(3)} \ar{d}{\text{\ref*{lem:ALSDetNKZeroDimKThenExistVec:part3}}} &\text{$\cdots$}  \\
\text{\ref*{lem:ALSDetNKZeroDimKThenExistVec:P2}(2)} \ar{ur}{\text{\ref*{lem:ALSDetNKZeroDimKThenExistVec:part2}}} & \text{\ref*{lem:ALSDetNKZeroDimKThenExistVec:P2}(3)} \ar{ur}{\text{\ref*{lem:ALSDetNKZeroDimKThenExistVec:part2}}} &\text{$\cdots$}
\end{tikzcd}
\end{center}
Thus, the principle of mathematical induction will imply the statement of the lemma for $k > 1$ if we prove the claims \ref*{lem:ALSDetNKZeroDimKThenExistVec:part1}, \ref*{lem:ALSDetNKZeroDimKThenExistVec:part2} and \ref*{lem:ALSDetNKZeroDimKThenExistVec:part3}.

\paragraph{\ref{lem:ALSDetNKZeroDimKThenExistVec:part1}} By contradiction. Let $n \ge k + 2 = 4$ and $\K \subseteq \M_{n\,k}(\mathbb F)$ be a linear variety with $\codim(\K) = k = 2$ such that $\det_{n\,2} (X) = 0$ for every $X \in \K$. Suppose that there exists a cobasis $B^*$ of  $\m(\K)$ and $1 \le j' \le 2$  such that  $B^* \cap \big([n]\times \{j'\}\big) = \varnothing$. Without loss of generality let us assume that $j' = 1$. Thus, 

\begin{equation}\label{lem:ALSDetNKZeroDimKThenExistVec:eq:1001}
B^* \cap \big([n]\times \{1\}\big) = \varnothing
\end{equation}

Assuming this, let us show that  $\{(1,2), (2,2)\}$ is a cobasis of $\m(\K)$. \linebreak If $B^* = \{(1,2), (2,2)\}$, then the claim follows immediately. Now consider the case if $B^* \neq \{(1,2), (2,2)\}$. Assume without loss of generality that $(1,2) \not\in B^*$. Then there exists $2 \le i' \le n$ such that $(i',2)\in B^*$. The equality~\eqref{lem:ALSDetNKZeroDimKThenExistVec:eq:1001} imply that
\begin{equation}\label{lem:ALSDetNKZeroDimKThenExistVec:eq:10001}
\left|B^* \cap \left(\{i'\}\times [2]\right)\right| = 1.
\end{equation}
Let $\K'$ be defined by
$$\K' = \pi_{\ins_{n\, 2}^{(i',1)}}\big(\K\cap U(x_{(1\,i)} = \delta_{i\,i'} \mid 1 \le i \le n)\big).$$
The equality~\eqref{lem:ALSDetNKZeroDimKThenExistVec:eq:1001} allows us to apply Lemma~\ref{lem:StrikingOutLifting} to $\K$, $i'$, $j' = 1$ and $\{c_i = \delta_{i\,i'} \mid 1 \le i \le n\}$. From this we conclude that $\K'$ is a linear variety,
\begin{equation}\label{lem:ALSDetNKZeroDimKThenExistVec:eq:1002}
\det_{(n-1)\,1}(X') = 0\;\;\mbox{for all}\;\;X' \in \K'
\end{equation}
and 
$$\codim(\K') \le \codim(\K) - \left|B^* \cap \left(\{i'\}\times [2]\right)\right| \overset{\eqref{lem:ALSDetNKZeroDimKThenExistVec:eq:10001}}{=\joinrel=\joinrel=}  \codim(\K) - 1 = 1.$$ 
Since $\det_{(n-1)\,1}$ is a nonzero function, then $\codim(\K') > 0$ and consequently 
\begin{equation}\label{lem:ALSDetNKZeroDimKThenExistVec:eq:1003}
\codim(\K') = 1.
\end{equation}

Let $A' \in \M_{1\, [n-1]\times \{1\}}(\F)$  be defined by $A'_{1\,(i,1)} = (-1)^{i-1}$ for all $1 \le i \le n-1$. Since $A'$ is nonzero,
\begin{equation}\label{lem:ALSDetNKZeroDimKThenExistVec:eq:1004}
\codim(\AS(A',\mathbf 0 )) = 1.
\end{equation}
In addition, the definition of Cullis' determinant (Definition~\eqref{def:CullisDet}) implies that 
$$\det_{(n-1)\, 1}(X) = \sum_{i = 1}^{n-1}(-1)^{i-1}x_{i\,1} =   A'X\;\;\mbox{for all}\;\;X = (x_{i\,j}) \in \M_{(n-1)\, 1},$$
where $X$ in $A'X$ is considered as a vector in $\F^{[n-1]\times \{1\}}$. Therefore, following~\eqref{lem:ALSDetNKZeroDimKThenExistVec:eq:1002}, we conclude that $\K' \subseteq \AS(A',\mathbf 0)$. Hence,
$\K' = \AS(A',\mathbf 0)$
by Lemma~\ref{lem:ASSubAS} because~\eqref{lem:ALSDetNKZeroDimKThenExistVec:eq:1003} and~\eqref{lem:ALSDetNKZeroDimKThenExistVec:eq:1004} imply that $\codim(\AS(A',\mathbf 0 )) = \codim(\K')$.

Since $A'_{1\,(1,1)} = 1 \neq 0$, then $\rk\Big(A'\big(\big|(1,1)\big]\Big) = 1$. Therefore, 
\begin{equation}\label{lem:ALSDetNKZeroDimKThenExistVec:eq:1005}
{B^*}' = \{(1,1)\}
\end{equation}
is a cobasis of $\K'$ by the definition of $\m[A']$ and Corollary~\ref{cor:MatMatDualToIndepMat}. By Lemma~\ref{lem:StrikingOutLifting}\ref{lem:StrikingOutLifting:part1} applied to $\K$, $i'$, $j' = 1$ and $\{c_i = \delta_{i\,i'} \mid 1 \le i \le n\}$,  the set
$${I^*} = \Big(B^* \cap \big(\{i'\} \times [k]\big)\Big) \cup \ins_{n\,2}^{(i',1)}({B^*}' )$$
 is a coindependent set of $\K$. Since
\[
\Big(B^* \cap \big(\{i'\} \times [k]\big)\Big) = \{(i',2)\}
\]
and
\[
\ins_{n\,2}^{(i',1)}({B^*}') \overset{\eqref{lem:ALSDetNKZeroDimKThenExistVec:eq:1005}}{=\joinrel=\joinrel=} \ins_{n\,2}^{(i',1)}(\{(1, 1)\}) = \{(1,2)\},
\]
then
\[
{I^*} = \{(i',2)\} \cup \{(1,2)\} = \{(1,2), (i',2)\}.
\]
Since $\codim(\K) = 2$, then ${I^*}$ is a cobasis of $\m(\K)$. If $i' = 2,$ then the claim is established. Otherwise, by setting ${B_{\mathrm{new}}^*} = I^*$, applying the arguments above to ${B_{\mathrm{new}}^*}$ and $i' = 1$ we obtain that ${B^C}^* = \{(1,2), (2,2)\}$ is a cobasis of $\K$. %{B_{\mathrm{desired}}^*} =

Thus, since ${B^C}^*$ is a cobasis of $\m(\K)$, then $B^C = \big([n]\times [2]\big)\setminus {B^C}^*$ is a basis of $\m(\K)$. Let $I^C \subseteq B^C, \mathbf c^C \in \F^{I^C}$ and $\K^C \subseteq \K$  be defined by
\begin{equation*}
I^C = \{(1,1), \ldots, (n,1), (5,2), \ldots, (n,2)\},
\end{equation*}
\begin{equation}\label{lem:ALSDetNKZeroDimKThenExistVec:eq1001}
\mathbf c^C_e =  
\begin{cases}
1, & e = (1,1)\\
1, & e = (2,1)\\
0, & e = (3,1)\\
-1, & e = (4,1)\\
0, & \mbox{otherwise}
\end{cases}\;\;\mbox{for all}\;\; e \in I^C
\end{equation} 
and 
\begin{equation*}
\K^C = \K \cap U(x_{e} = \mathbf c^C_e\;|\; e \in I^C).
\end{equation*} 
Since $I^C \subseteq B^C$, then $I^C$ is an independent set of $\K$.

Lemma~\ref{lem:LinVarSec} applied to $\K$, $I_C$ and $\mathbf c^C$ implies that 
$\K^C$ is a linear variety and 
\begin{equation}\label{lem:ALSDetNKZeroDimKThenExistVec:eq101}
\m(\K^C)\setminus I^C = \m(\K)/I^C.
\end{equation}

Let $I^{(3,2)} = \{(3,2)\} \cup I^C$. Since $I^{(3,2)} \subseteq B^C$, it is an independent set of $\m(\K)$. Hence, $\{(3,2)\} = I^{(3,2)} \setminus I^C$ is an independent set of $\m(\K)/I^C$ by Corollary~\ref{cor:ContractionIndependent}. The equality~\eqref{lem:ALSDetNKZeroDimKThenExistVec:eq101} implies that $\{(3,2)\}$ is an independent set of $\m(\K^C)\setminus I^C$. By the definition of restriction matroid (Definition~\ref{def:RestrictionMatroid}), $\{(3,2)\}$ is an independent set of $\m(\K^C)$. From this and Corollary~\ref{cor:IndepEveryVal} applied to $\K^C$, $I = \{(3,2)\}$ and $\mathbf c = (1) \in \F^{\{(3,2)\}}$ we conclude that there exists $X^C = (x^C_{i\,j}) \in \K^C$ such that 
\begin{equation}\label{lem:ALSDetNKZeroDimKThenExistVec:eq102}
x^C_{3\,2} = 1.
\end{equation}
Considering $\det_{n\,2} (X^C)$, we have on the one hand that 
\begin{equation}\label{lem:ALSDetNKZeroDimKThenExistVec:eq1002}
\det_{n\,2} (X^C) = 0
\end{equation}
because  $\K^C \subseteq \K$ and  $\det_{n\,2} (X) = 0$ for all $X \in \K^C$ by the initial assumption on $\K$. 

On the other hand, if we express $\det_{n\,2}(X^C)$ in terms of $x^C_{i\,j}$ using  Laplace expansion along the first column, by Lemma~\ref{lem:DetNKLaplaceExp}  we obtain that
\begin{multline}\label{lem:ALSDetNKZeroDimKThenExistVec:eq1003}
\det_{n\,2}(X^C) = \sum_{i = 1}^n (-1)^{i+1} x^C_{i\, 1} \det_{(n-1)\, 1}\left(X^C(i|1)\right)\\
  \overset{\eqref{lem:ALSDetNKZeroDimKThenExistVec:eq1001}}{=\joinrel=\joinrel=} 1\cdot \det_{(n-1)\,1}\left(X^C(1|1)\right) - 1\cdot \det_{(n-1)\,1}\left(X^C(2|1)\right) + 0 \cdot \det_{(n-1)\,1}\left(X^C(3|1)\right)\\
\phantom{XX}  - (-1)\cdot \det_{(n-1)\,1}\left(X^C(4|1)\right) + \sum_{i = 5}^n (-1)^{i+1} 0\cdot \det_{(n-1)\,1}\left(X^C(i|j)\right)\\
\phantom{XXXXXX} = (x^C_{2\,2} - x^C_{3\,2} + x^C_{4\,2} + 0 + \ldots + 0) - (x^C_{1\,2} - x^C_{3\,2} + x^C_{4\,2} + 0 + \ldots + 0)\\  + (x^C_{1\,2} - x^C_{2\,2} + x^C_{3\,2} + 0 + \ldots + 0)
 = x^C_{3\,2}.
\end{multline}
Aligning~\eqref{lem:ALSDetNKZeroDimKThenExistVec:eq1003} and~\eqref{lem:ALSDetNKZeroDimKThenExistVec:eq102} together yields $\det_{n\,2}(X^C) = 1$. This contradicts with~\eqref{lem:ALSDetNKZeroDimKThenExistVec:eq1002}. Thus, this part of the lemma is established.

\paragraph{\ref{lem:ALSDetNKZeroDimKThenExistVec:part2}}
%Suppose that $k > 2$ and~\ref{lem:ALSDetNKZeroDimKThenExistVec:P2} holds for $k-1$. Let us prove~\ref{lem:ALSDetNKZeroDimKThenExistVec:P1} for $k$ 
By contradiction. Assume that $k \ge 3$ and~\ref{lem:ALSDetNKZeroDimKThenExistVec:P2} holds for $k-1$. Suppose that a linear variety $\K \subseteq \M_{n\,k}(\mathbb F)$, a cobasis $B^*$ of $\m(\K)$ and  $1 \le j' \le k$ are such that $\codim(\K) = k$ $\det_{n\,k} (X) = 0$ for all $X \in \K$ and $B^* \cap \big([n]\times \{j'\}\big) = \varnothing$. 

Since $\left|B^*\right| = k > 0$, there exists $1 \le i' \le n$ such that $\Big|B^* \cap \big(\{i'\} \times [k]\big)\Big| > 0$. Lemma~\ref{lem:ALSDetNKZeroCapLE1} implies that $\Big|B^* \cap \big(\{i'\} \times [k]\big)\Big| \le 1$. Therefore, 
\begin{equation}\label{lem:ALSDetNKZeroDimKThenExistVec:eq1103}
\Big|B^* \cap \big(\{i'\} \times [k]\big)\Big| = 1.
\end{equation}
 Let $\K' = \pi_{\ins_{n\,k}^{(i',j')}}\Big(\K \cap U\big((i,j') = \delta_{i\,i'}\mid 1 \le i \le n\big)\Big)$. From  Lemma~\ref{lem:StrikingOutLifting} applied to $i'$, $j'$, $\K$, $B^*$  and $\{c_i = \delta_{i\,i'}\mid 1 \le i \le n\}$ we conclude that $\K'$ is a linear variety, 
 \begin{equation}\label{lem:ALSDetNKZeroDimKThenExistVec:eq1105}
 \det_{(n-1)\,(k-1)}(X') = 0\;\;\mbox{for all}\;\;X' \in \K'
 \end{equation}
 and
 \begin{equation}\label{lem:ALSDetNKZeroDimKThenExistVec:eq1104}
 \codim(\K') \le\codim(\K) - \Big|B^* \cap \big(\{i'\} \times [k]\big)\Big|  \overset{\eqref{lem:ALSDetNKZeroDimKThenExistVec:eq1103}}{=\joinrel=\joinrel=} \codim(\K) - 1 = k - 1.
\end{equation}
The condition~\eqref{lem:ALSDetNKZeroDimKThenExistVec:eq1105} allows us to apply Theorem~\ref{thm:ALSDetNKZeroCodimGEK} to $\K'$. It follows that
\begin{equation}\label{lem:ALSDetNKZeroDimKThenExistVec:eq1106}
\codim(\K') \ge k - 1.
\end{equation}
Thus, 
\[
k - 1 \overset{\eqref{lem:ALSDetNKZeroDimKThenExistVec:eq1106}}{\le} \codim(\K') \overset{\eqref{lem:ALSDetNKZeroDimKThenExistVec:eq1104}}{\le} k - 1
\]
and consequently 
\begin{equation}\label{lem:ALSDetNKZeroDimKThenExistVec:eq1107}
\codim(\K') = k - 1.
\end{equation}
Since we assume that~\ref*{lem:ALSDetNKZeroDimKThenExistVec:P2} holds for $k-1$, applying~\ref*{lem:ALSDetNKZeroDimKThenExistVec:P2} to $\K'$ yields that there is $0 \neq \mathbf z' \in \F^{n-1}$ such that
\begin{equation}\label{lem:ALSDetNKZeroDimKThenExistVec:eq1}
\K' = \{X' \in \M_{n\,k}(\F) \mid \mathbf {z'}^t X' = 0\}.
\end{equation}
Let us assume without loss of generality that there is $1 \le i_0 \le n-1$ such that
\begin{equation}\label{lem:ALSDetNKZeroDimKThenExistVec:eq1101}
\mathbf z'_{i_0} = 1.
\end{equation}

Let  $A' \in \M_{(k-1)\, [n-1]\times[k-1]}(\F)$ be defined by
\begin{equation}\label{lem:ALSDetNKZeroDimKThenExistVec:eq1102}
A'_{p\,(i,j)} =
\begin{cases}
\mathbf z'_{i}, & j = p\\
0, & \mbox{otherwise}
\end{cases}
\end{equation}
for all $1 \le p \le k-1$ and $(i,j) \in [n-1]\times[k-1]$. Then  $A'\big(\big|(i_0,1),\ldots, (i_0,k-1)\big]$ is a permutation matrix by~\eqref{lem:ALSDetNKZeroDimKThenExistVec:eq1101} and~\eqref{lem:ALSDetNKZeroDimKThenExistVec:eq1102} and consequently
\begin{equation}\label{lem:ALSDetNKZeroDimKThenExistVec:eq1109}
\rk\bigg(A'\big(\big|(i_0,1),\ldots, (i_0,k-1)\big]\bigg) = k - 1.
\end{equation}
Since $A'$ has $k-1$ rows, from~\eqref{lem:ALSDetNKZeroDimKThenExistVec:eq1109} we obtain that
\begin{equation}\label{lem:ALSDetNKZeroDimKThenExistVec:eq1108}
\rk(A') = k-1.
\end{equation}
Let ${B'}^* \subseteq \left([n-1]\times[k-1]\right)$ be defined by
\begin{equation}\label{lem:ALSDetNKZeroDimKThenExistVec:eq2}
{B'}^* = \{(i_0,1),\ldots, (i_0,k-1)\}.
\end{equation}
Using the definition of vector matroid $\m[A']$ (Definition~\ref{def:VecMat}), we conclude from equalities~\eqref{lem:ALSDetNKZeroDimKThenExistVec:eq1109} and~\eqref{lem:ALSDetNKZeroDimKThenExistVec:eq1108} that
\begin{equation}\label{lem:ALSDetNKZeroDimKThenExistVec:eq1110}
{B'}^*\in \mathcal B(\m[A']).
\end{equation}

The equality~\eqref{lem:ALSDetNKZeroDimKThenExistVec:eq1} means that $\K' \subseteq \AS(A',0)$. Hence, $\K' = \AS(A',0)$ by Lemma~\ref{lem:ASSubAS} because 
$$\codim(\K') \overset{\eqref{lem:ALSDetNKZeroDimKThenExistVec:eq1107}}{=\joinrel=} k-1 \overset{\eqref{lem:ALSDetNKZeroDimKThenExistVec:eq1108}}{=\joinrel=} \rk(A') \overset{\mbox{\scriptsize Lemma~\ref{lem:MatRankEqToCodim}}}{=\joinrel=\joinrel=\joinrel=\joinrel=\joinrel=\joinrel=} \codim(\AS(A',0)).$$ Therefore, ${B'}^*$ is a cobasis of $\m(\K')$ by Lemma~\ref{lem:IdepMatroidCorr} and~\eqref{lem:ALSDetNKZeroDimKThenExistVec:eq1110}.

 By Lemma~\ref{lem:StrikingOutLifting}\ref{lem:StrikingOutLifting:part1}, $I^* = \Big(B^* \cap \big(\{i'\} \times [k]\big)\Big) \cup \ins_{n\,k}^{(i',j')}({B'}^*)$ is a coindependent set of $\m(\K)$. Hence, $I^*$ is a cobasis of $\K$ because
\begin{multline*}
\bigg|\Big(B^* \cap \big(\{i'\} \times [k]\big)\Big) \cup \ins_{n\,k}^{(i',j')}({B'}^*)\bigg|\\
 = \Big|B^* \cap \big(\{i'\} \times [k]\big)\Big| + \Big|\ins_{n\,k}^{(i',j')}({B'}^*)\Big| \overset{\eqref{lem:ALSDetNKZeroDimKThenExistVec:eq1103}}{=\joinrel=\joinrel=} 1 + \Big|\ins_{n\,k}^{(i',j')}({B'}^*)\Big|\\
 \overset{\eqref{lem:ALSDetNKZeroDimKThenExistVec:eq2}}{=\joinrel=\joinrel=\joinrel=}
1 + (k-1) = k = \codim(\K).
\end{multline*}
In addition, $I^*$ contains $k-1 \ge 3 - 1 =  2$ elements of the set
$\ins_{n\,k}^{(i',j')}({B'}^*)$. The equality~\eqref{lem:ALSDetNKZeroDimKThenExistVec:eq2} implies that all elements of $\ins_{n\,k}^{(i',j')}({B'}^*)$ belong to the same row by Lemma~\ref{lem:InsPresRowCol}. This leads to a contradiction because by Lemma~\ref{lem:ALSDetNKZeroCapLE1} every cobasis of $\K$ has no more than one element in every row of $[n]\times [k]$.

\paragraph{\ref{lem:ALSDetNKZeroDimKThenExistVec:part3}}
Let $\K \subseteq \M_{n\,k}(\mathbb F)$ be a linear variety with $\codim(\K) = k$ such that $\det_{n\,k} (X) = 0$ for all $X \in \K$, $B^*$ be a cobasis of $\m(\K)$. Since~\ref*{lem:ALSDetNKZeroDimKThenExistVec:P1} holds for $k$ and $\left|B^*\right| = \codim(\K) = k$, then $B^*$ has exactly one element in each column of $[n]\times [k]$. 

By Lemma~\ref{lem:ForLinVarExistsEqRep}, there exist $m \ge 1$, $A \in \M_{m\,\left([n]\times [k]\right)}$ and $\mathbf b \in \F^{m}$ such that $\K = \AS(A,\mathbf b)$. In addition, Lemma~\ref{lem:ReducedERExists} allows us to assume that is $A$ reduced with respect to $B^*$ and consequently $m = k$. For every $1 \le p \le k$ let us denote by $a_p$ a unique element of $B^*$ such that 
\begin{equation}\label{lem:ALSDetNKZeroDimKThenExistVec:eq3}
A_{p\,a_p} = 1.
\end{equation}
Since the permutation of rows of system of linear equations does not affect on the set of its solutions, we also assume that $a_p$ belongs to $[n]\times \{p\}$ for every $1 \le p \le k$.

Let us show that 
\begin{equation}\label{lem:ALSDetNKZeroDimKThenExistVec:eq203}
A_{p\,a} = 0\;\;\mbox{for all}\;\;a \in \big([n]\times [k]\big) \setminus \big([n]\times \{p\}\big).
\end{equation}
Indeed, let $a \in \big([n]\times [k]\big)$ be such that $A_{p\,a} \neq 0$. Then ${B'}^* = B^* \triangle \{a, a_p\}$ is a cobasis of $\m(\K)$ by Lemma~\ref{lem:ALSReducedBNonzeroExchange}. Therefore, 
\begin{equation}\label{lem:ALSDetNKZeroDimKThenExistVec:eq201}
{B'}^* \cap \big([n] \times \{p\}\big) \neq \varnothing
\end{equation}
by \ref*{lem:ALSDetNKZeroDimKThenExistVec:P1} for $k$ applied to $\K$ and ${B'}^*$. In addition,
\begin{multline}\label{lem:ALSDetNKZeroDimKThenExistVec:eq202}
{B'}^* \cap \big([n] \times \{p\}\big) = \big(B^* \triangle \{a, a_p\}\big) \cap \big([n] \times \{p\}\big)\\
 = \Big(B^*  \cap \big([n] \times \{p\}\big) \Big) \triangle \left(\{a, a_p\} \cap \Big([n] \times \{p\}\Big) \right)\\
  = \Big(\{a_p\} \cap \big([n] \times \{p\}\big) \Big) \triangle  \Big(\{a, a_p\} \cap \big([n] \times \{p\}\big) \Big)\\
    = \big(\{a_p\} \triangle \{a, a_p\}\big) \cap \big([n] \times \{p\}\big)  = \{a\} \cap \big([n] \times \{p\}\big).
\end{multline}
By aligning~\eqref{lem:ALSDetNKZeroDimKThenExistVec:eq201} and~\eqref{lem:ALSDetNKZeroDimKThenExistVec:eq202} together we obtain that $\{a\} \cap \big([n] \times \{p\}\big) \neq \varnothing$, or, equivalently, $a \in \big([n] \times \{p\}\big)$. This establishes~\eqref{lem:ALSDetNKZeroDimKThenExistVec:eq203} which allows us to conclude that $A$ has the form
\[
\begin{pmatrix}
A_{1\, (1,1)} & \cdots & A_{1\, (n,1)} & 0 & \cdots & 0 & \cdots & 0              & \cdots & 0 \\
0       & \cdots & 0        & A_{2\, (1,2)} & \cdots & A_{2\, (n,2)} & \cdots      & 0              & \cdots & 0\\
\vdots  & \ddots & \vdots   & \vdots        & \ddots & \vdots        & \ddots      & \vdots         & \ddots & \vdots\\
0       & \cdots & 0        & 0             & \cdots & 0             & \cdots      &  A_{k\, (1,k)} & \cdots & A_{k\, (n,k)}
\end{pmatrix}
\]

Let $\mathbf w^{(1)}, \ldots \mathbf w^{(p)} \in \F^{n}$ be defined by
\[
\mathbf w^{(p)} = (A_{p\, (1,p)}, \ldots, A_{p\, (n,p)}).
\]
All these vectors are nonzero because 
\begin{equation}\label{lem:ALSDetNKZeroDimKThenExistVec:eq4}
\mathbf w^{(p)}_{a_p} = 1\;\mbox{for all}\;1 \le p \le k
\end{equation} by the equality~\eqref{lem:ALSDetNKZeroDimKThenExistVec:eq3}. 

Let us prove that $\mathbf w^{(1)}, \ldots, \mathbf w^{(k)}$ are scalar multiples of each other by contradiction. Suppose  that there is  $1 \le p_1 < p_2 \le k$ such that  $\mathbf w^{(p_2)}$ is not a scalar multiple of $\mathbf w^{(p_1)}$. Consider a linear map $T \colon \M_{n\,k}(\F) \to \M_{n\,k}(\F)$ sending $X \in \M_{n\,k}(\F)$ to $X$ with $p_1$-th column subtracted from its $p_2$-th column. This map is invertible and preserves $\det_{n\,k}$ by the property~\ref{thm:DetNKBasProp:mark1} in Theorem~\ref{thm:DetNKBasProp}. Let $D \in \M_{[n]\times [k]\,[n]\times [k]}(\F)$  be a matrix representing $T$ in the standard basis of $\M_{n\, k}(\F)$ identified with $\F^{[n]\times [k]}$. Then by direct calculations we obtain from the definition of $T$ that
\[D^{-1} = I_{[n]\times [k]} + \sum_{i=1}^n E_{(i,p_2)\,(i,p_1)}.\]

Thus, $\K_T = T(\K)$ is a linear variety of codimension $k$ such that $\det_{n\,k}(X) = 0$ for all $X \in \K_T$. In addition, $\K = \AS(AD^{-1}, \mathbf b)$ by Lemma~\ref{lem:InvertibleLDAS}.

The matrix $AD^{-1}$ could be represented as follows
\[
\begin{pmatrix}
A_{p_1\, (1,1)} & \cdots & A_{p_1\, (n,1)} & 0 & \cdots & 0 & \cdots & 0              & \cdots & 0 \\
A_{p_2\, (1,1)} & \cdots & A_{p_2\, (n,1)}  & A_{p_2\, (1,2)} & \cdots & A_{p_2\, (n,2)} & \cdots      & 0              & \cdots & 0\\
\vdots  & \ddots & \vdots   & \vdots        & \ddots & \vdots        & \ddots      & \vdots         & \ddots & \vdots\\
0       & \cdots & 0        & 0             & \cdots & 0             & \cdots      &  A_{k\, (1,2)} & \cdots & A_{k\, (n,2)}
\end{pmatrix},
\]
where we put $p_1$-th and $p_2$-th row of $AD^{-1}$ first for convenience. If we subtract the $p_1$-th row of $AD^{-1}$ with the coefficient $\alpha = \frac{\mathbf w^{(p_2)}_{a_{p_2}}}{\mathbf w^{(p_1)}_{a_{p_1}}}$  from its $p_2$-th row, we obtain the matrix $A'$ which has the form
\[
\begin{psmallmatrix}
A_{p_1\, (1, p_1)} & \cdots & A_{p_1\,a_{p_1}} & \cdots & A_{p_1\, (n, p_1)} & 0 & \cdots & 0 & \cdots & 0 & \cdots & 0 \\
A_{p_2\, (1, p_2)} - \alpha A_{p_1\, (1,p_1)} & \cdots & 0 & \cdots& A_{p_2\, (n, p_2)} - \alpha A_{p_1\, (n, p_1)} & A_{p_2\, (1,p_2)} & \cdots & A_{p_2\, (n,p_2)} & \cdots      & 0              & \cdots & 0\\
\vdots  & \ddots & \vdots & \ddots & \vdots   & \vdots        & \ddots & \vdots        & \ddots      & \vdots         & \ddots & \vdots\\
0       & \cdots & 0 & \cdots       & 0             &  0             & \cdots     & 0 & \cdots &  A_{k\, (1,k)} & \cdots & A_{k\, (n,k)}
\end{psmallmatrix}.
\]
Therefore, $A'$ is reduced with respect to $B^* = \{a_{p_1}, \ldots, a_{p_k}\}$. Hence, $B^*$ is a cobasis of $\m(\K_T)$.

Since $\mathbf w^{(p_2)}$ is not scalar multiple of $\mathbf w^{(p_1)}$, $A'_{p_2\, a}$ is nonzero for some $a \in [n]\times \{p_1\}$. Therefore, Lemma~\ref{lem:ALSReducedBNonzeroExchange} implies that $B^* \triangle \{a_{p_2}, a\}$ is a cobasis of $\m(\K_T)$ which does not contain an element belonging to $[n]\times \{p_2\}$. This lead us to a contradiction because we assume that~\ref*{lem:ALSDetNKZeroDimKThenExistVec:P1} holds for $k$.

Thus, $\mathbf w^{(1)}, \ldots, \mathbf w^{(k)}$ are proportional to each other. In addition, all $\mathbf w^{(p)}$ are nonzero by~\eqref{lem:ALSDetNKZeroDimKThenExistVec:eq4}. Hence, there exists $\mathbf 0 \neq \mathbf z\in \F^n$ such that $\mathbf w^{(p)} = \alpha_p \mathbf z$ for some $0 \neq \alpha_1, \ldots, \alpha_p \in \F$ and all $1 \le p \le k$. Since the multiplication of the row of system of linear equations by a nonzero scalar does not change the set of its solutions,
\begin{equation}\label{lem:ALSDetNKZeroDimKThenExistVec:eq605}
\K = \AS(A', \mathbf b')
\end{equation}
for $A' = \diag(\alpha^{-1}_1, \ldots, \alpha^{-1}_k) A$ and $\mathbf b' = \diag(\alpha^{-1}_1, \ldots, \alpha^{-1}_k)\mathbf b$,
which is equivalent to
\begin{equation}\label{lem:ALSDetNKZeroDimKThenExistVec:eq5}
\K = \left\{X \in \M_{n\,k}(\F)\,\mid\,\mathbf z^t X = \mathbf {b'}\right\}.
\end{equation}

Let us prove that $\mathbf b' = \mathbf 0$ by contradiction. Suppose that  $\mathbf b' \neq \mathbf 0$ and assume without loss of generality that 
\begin{equation}\label{lem:ALSDetNKZeroDimKThenExistVec:eq601}
\mathbf b'_1 \neq 0.
\end{equation}

Let $1 \le i_0 \le n$ be such that 
\begin{equation}\label{lem:ALSDetNKZeroDimKThenExistVec:eq604}
\mathbf z_{i_0} \neq 0.
\end{equation} 
From the definition of vector matroid (Definition~\ref{def:VecMat}) $\m[A']$ we conclude by~\eqref{lem:ALSDetNKZeroDimKThenExistVec:eq604} that the set $B^* = \{i_0\} \times [k]$ is a basis of $\m[A']$. Hence, $B = \big([n]\times [k]\big) \setminus \big(\{i_0\} \times [k]\big)$ is a basis of $\m(\K)$ by Lemma~\ref{lem:IdepMatroidCorr}\ref{lem:IdepMatroidCorr:part2} and~\eqref{lem:ALSDetNKZeroDimKThenExistVec:eq605}. Therefore, by Corollary~\ref{cor:IndepEveryVal}, there is $X' = (x'_{i\,j})\in \K$ such that 
\begin{equation}\label{lem:ALSDetNKZeroDimKThenExistVec:eq6}
x'_{i\,1} = 0\;\mbox{for all}\; 1 \le i \neq i_0 \le n
\end{equation}
and
\begin{equation}\label{lem:ALSDetNKZeroDimKThenExistVec:eq602}
\det_{(n-1)\,(k-1)}\Bigl(X'(i_0|1)\Bigr) \neq 0
\end{equation}
 because the matrix $X'(i_0|1)$ could be arbitrary and $\det_{(n-1)\, (k-1)}$ is clearly a nonzero function. In addition,
$$
\mathbf z_{i_{0}} x'_{i_0\,1} \overset{\eqref{lem:ALSDetNKZeroDimKThenExistVec:eq6}}{=\joinrel=\joinrel=}  \sum_{i = 1}^{n}\mathbf z_i x'_{i\,1} \overset{\eqref{lem:ALSDetNKZeroDimKThenExistVec:eq5}}{=\joinrel=\joinrel=} \mathbf b'_1 \overset{\eqref{lem:ALSDetNKZeroDimKThenExistVec:eq601}}{=\joinrel\neq\joinrel=} 0$$
and consequently, by~\eqref{lem:ALSDetNKZeroDimKThenExistVec:eq604} we have
\begin{equation}\label{lem:ALSDetNKZeroDimKThenExistVec:eq7}
x'_{i_0\,1} \neq 0.
\end{equation} 

Let us consider $\det_{n\,k}(X')$. On the one hand, since $X' \in \K$, then from the initial assumption on $\K$ we obtain that
 \begin{equation}\label{lem:ALSDetNKZeroDimKThenExistVec:eq603}
\det_{n\,k}(X') = 0.
\end{equation}

On the other hand, using the Laplace expansion of $\det_{n\,k}(X')$ along the first row\linebreak (Lemma~\ref{lem:DetNKLaplaceExp}), we obtain
\begin{multline*}
\det_{n\,k}(X') = \sum_{i = 1}^{n} (-1)^{i + 1}x'_{i\,1}\cdot \det_{(n-1)\,(k-1)} \Bigl(X'(i|1)\Bigr) \\ \overset{\eqref{lem:ALSDetNKZeroDimKThenExistVec:eq6}}{=\joinrel=\joinrel=} x'_{i_0\, 1}\cdot \det_{(n-1)\,(k-1)} \Bigl(X'(i|1)\Bigr) \overset{\eqref{lem:ALSDetNKZeroDimKThenExistVec:eq7}\;\mbox{\scriptsize and}\;\eqref{lem:ALSDetNKZeroDimKThenExistVec:eq602}}{=\joinrel=\joinrel=\joinrel\neq\joinrel=\joinrel=\joinrel=} 0
\end{multline*}
which contradicts with~\eqref{lem:ALSDetNKZeroDimKThenExistVec:eq603}. Therefore, $\mathbf b' = \mathbf 0$.
\end{proof}

\begin{remark}If $k > 1$, then the statement of the lemma is not true for $n = k$ and $n = k + 1$. Indeed, if $n = k$, then $\det_{n\,k}(X) = \det_{k} (X) = 0$ for all $X \in \K'$, where
\[
\K' = \{\begin{pmatrix}
0 & x_{1,1} & \cdots & x_{1\,n-1}\\
\vdots & \vdots & \ddots & \vdots\\
0 & x_{n,1} & \cdots & x_{n\,n-1}
\end{pmatrix} \mid x_{1\,1},\ldots, x_{n\,n-1} \in \F\}.
\]
If $n=k+1$, then $\det_{n\,k}(X) = \det_{k} (X) = 0$ for all $X \in \K''$, where
 \[
\K'' = \{\begin{pmatrix}
x & x_{1,1} & \cdots & x_{1\,n-2}\\
\vdots & \vdots & \ddots & \vdots\\
x & x_{n,1} & \cdots & x_{n\,n-2}
\end{pmatrix} \mid x, x_{1\,1},\ldots, x_{n\,n-2} \in \F\}
\]
which follows from multilinearity of $\det_{n\,k}$ with respect to columns and\linebreak Lemma~\ref{lem:NAKAGAMILEMMA20} because $n + k - 1$ is even in this case.
\end{remark}

\begin{lemma}\label{thm:RowVarDetNKZeroExpress}
Let $\mathbf z \in \F^n$ be such that $\mathbf z_1 = -1$ and $\K \subseteq \M_{n\,k}(\mathbb F)$ be a linear variety defined by $\K = \left\{X \in \M_{n\,k}(\F) \mid \mathbf z^t X = \mathbf 0\right\}$. Then 
\begin{equation}\label{thm:RowVarDetNKZeroExpress:eqq1}
\det_{n\,k}(X) = \sum_{\substack{c \in \binom{[n]}{k}\\1 \not\in c}} \left(1 + \sum_{\alpha=1}^{k}\mathbf z_{c(\alpha)} (-1)^{\alpha - c(\alpha)}\right)\sgn_{[n]}(c)\det_k\Bigl(X[c|)\Bigr)
\end{equation}
\end{lemma}
\begin{proof}Since $\mathbf z_1 = -1,$ the condition $\mathbf z^t X = 0$ for all $X \in \K$ means that the first row of every $X \in \K$ could be expressed as a linear combination of other rows as follows
\begin{equation}\label{thm:RowVarDetNKZeroCond:eq2}
X[1|) = \mathbf z_2 X[2|) + \ldots + \mathbf z_n X[n|).
\end{equation}

Let $X \in \K$. By Lemma~\ref{cor:CullisAsSumDet}
$$
\det_{n\, k} (X) = \sum_{c \in \binom{[n]}{k}} \sgn_{[n]}(c) \det_k\Bigl(X[c|)\Bigr).
$$
The right-hand side of this equality is divided into two summands as follows.
\begin{equation}\label{thm:RowVarDetNKZeroCond:eq3}\sum_{c \in \binom{[n]}{k}} \sgn_{[n]}(c) \det_k\Bigl(X[c|)\Bigr)\\
= \sum_{\substack{c \in \binom{[n]}{k}\\1 \in c}} \sgn_{[n]}(c) \det_k\Bigl(X[c|)\Bigr) + \sum_{\substack{c \in \binom{[n]}{k}\\1 \not\in c}} \sgn_{[n]}(c) \det_k\Bigl(X[c|)\Bigr).
\end{equation}

Consider the first summand of the right-hand side of~\eqref{thm:RowVarDetNKZeroCond:eq3}. Let $c' \in \binom{[n]}{k}$ be such that $1 \in c'$. It follows from~\eqref{thm:RowVarDetNKZeroCond:eq2} that 
\begin{equation}\label{thm:RowVarDetNKZeroCond:eq4}
 \det_k(X[c'|)) =
 \begin{vmatrix}
 X[1|)\\
 X[c'(2)|)\\
 \vdots\\
 X[c'(k)|)\\
 \end{vmatrix}
  = \begin{vmatrix}
 \mathbf z_2 X[2|) + \ldots + \mathbf z_n X[n|)\\
 X[c'(2)|)\\
 \vdots\\
 X[c'(k)|)\\
 \end{vmatrix}\\
  = \sum_{2 \le i \le n}\mathbf z_i\begin{vmatrix}
  X[i|)\\
 X[c'(2)|)\\
 \vdots\\
 X[c'(k)|)\\
 \end{vmatrix}
\end{equation}
Note that
\begin{equation}\label{thm:RowVarDetNKZeroCond:eq5}
\begin{vmatrix}
  X[i|)\\
 X[c'(2)|)\\
 \vdots\\
 X[c'(k)|)\\
 \end{vmatrix} = 0\;\;\mbox{for all}\;\;i \in c',\, 2 \le i \le n 
\end{equation}
because it becomes a determinant of a matrix with two rows equal to each other.

Since the ordinary determinant is an antisymmetric function of the rows of matrix, for $i \not\in c'$, $2 \le i \le n$  we have
\begin{equation}\label{thm:RowVarDetNKZeroCond:eq6}
\begin{vmatrix}
  X[i|)\\
 X[c'(2)|)\\
 \vdots\\
 X[c'(k)|)\\
 \end{vmatrix} = \sgn_{n\,k}(\sigma_{i\,c'})\det_k\Bigl(X[c'\triangle \{1,i\}|)\Bigr),
\end{equation}
where $\sigma_{i\, c'} \in \mathcal C_{[n]}^k$ is defined by
\[
\sigma_{i\,c'}(k) =
\begin{cases}
i, & k = 1\\
c'(k), & \mbox{otherwise}
\end{cases}.
\]

By substituting~\eqref{thm:RowVarDetNKZeroCond:eq5} and \eqref{thm:RowVarDetNKZeroCond:eq6} into~\eqref{thm:RowVarDetNKZeroCond:eq4} we obtain that
\begin{equation*}
\sum_{\substack{c' \in \binom{[n]}{k}\\1 \in c'}} \sgn_{[n]}(c') \det_k\Bigl(X[c'|)\Bigr)
= \sum_{\substack{c' \in \binom{[n]}{k}\\1 \in c'}} \sgn_{[n]}(c') \sum_{\substack{2 \le i \le n\\i \not\in c'}} \mathbf z_i\sgn_{n\,k}(\sigma_{i\,c'})\det_k\Bigl(X[c'\triangle \{1,i\}|)\Bigr).
\end{equation*}

Let us rearrange the sum in the right-hand side of this equality. Note that if $2 \le i \le n$ and $c'\triangle \{1,i\} = c$, then $i = c(\alpha)$ for some unique $1 \le \alpha \le k$. Hence, $\sgn_{[n]}(c') = (-1)^{1 - i}\sgn_{[n]}(c) = (-1)^{1 - c(\alpha)}\sgn_{[n]}(c)$ by the definition of $\sgn_{[n]}$. In addition, $\sgn_{n\,k}(\sigma_{i\,c'}) = (-1)^{\alpha - 1}$ because it is equal to the sign of permutation
\[
\begin{pmatrix}
c(1)        & c(2) & \cdots & c(\alpha) & c(\alpha+1) & \cdots & c(k)\\
c(\alpha)   & c(1) & \cdots & c(\alpha-1) & c(\alpha+1) & \cdots & c(k)
\end{pmatrix}
\]
by the definition. 

Observe also that for every pair $(c, \alpha)$ such that $c \in \binom{[n]}{k}, 1 \not\in c$ and $1 \le \alpha \le k$ there exist a unique pair $(c', i)$, where $c' \in \binom{[n]}{k}$ and $2 \le i \le n$ such that $i = c(\alpha)$, $c' = c \triangle \{1,i\}$ and consequently $1 \in c'$ and  $i \not\in c'$. 

Thus,
\begin{multline*}
\sum_{\substack{c' \in \binom{[n]}{k}\\1 \in c'}} \sgn_{[n]}(c') \sum_{\substack{2 \le i \le n\\i \not\in c'}} \mathbf z_i\sgn_{n\,k}(\sigma_{i\,c'})\det_k\Bigl(X[c'\triangle \{1,i\}|)\Bigr)\\
= \sum_{\substack{c' \in \binom{[n]}{k}, 2 \le i \le n\\1 \in c', i \not\in c'}} \mathbf z_i \sgn_{[n]}(c') \sgn_{n\,k}(\sigma_{i\,c'})\det_k\Bigl(X[c'\triangle \{1,i\}|)\Bigr)\\
= \sum_{\substack{c \in \binom{[n]}{k}, 1 \le \alpha \le k\\1 \not\in c}} \mathbf z_{c(\alpha)} (-1)^{1 - c(\alpha)}\sgn_{[n]}(c)(-1)^{\alpha - 1}\det_k\Bigl(X[c|)\Bigr)\\
= \sum_{\substack{c \in \binom{[n]}{k}\\1 \not\in c}} \left(\sum_{\alpha=1}^{k}\mathbf z_{c(\alpha)} (-1)^{\alpha - c(\alpha)}\right)\sgn_{[n]}(c)\det_k\Bigl(X[c|)\Bigr)
\end{multline*}

By substituting this instead of the first summand into~\eqref{thm:RowVarDetNKZeroCond:eq3} we obtain that
\begin{multline}\label{thm:MaxDimKEvenAltSumZero:eq2}
\det_{n\,k}(X) = \sum_{c \in \binom{[n]}{k}} \sgn_{[n]}(c) \det_k\Bigl(X[c|)\Bigr)\\
= \sum_{\substack{c \in \binom{[n]}{k}\\1 \not\in c}} \left(\sum_{\alpha=1}^{k}\mathbf z_{c(\alpha)} (-1)^{\alpha - c(\alpha)}\right)\sgn_{[n]}(c)\det_k\Bigl(X[c|)\Bigr) + \sum_{\substack{c \in \binom{[n]}{k}\\1 \not\in c}} \sgn_{[n]}(c) \det_k\Bigl(X[c|)\Bigr)\\
= \sum_{\substack{c \in \binom{[n]}{k}\\1 \not\in c}} \left(1 + \sum_{\alpha=1}^{k}\mathbf z_{c(\alpha)} (-1)^{\alpha - c(\alpha)}\right)\sgn_{[n]}(c)\det_k\Bigl(X[c|)\Bigr)
\end{multline}
for all $X \in \K$ which establishes the equality~\eqref{thm:RowVarDetNKZeroExpress:eqq1}
\end{proof}

\begin{lemma}\label{thm:RowVarDetNKZeroCond}
Let $\mathbf z \in \F^n$ be such that $\mathbf z_1 = -1$ and $\K \subseteq \M_{n\,k}(\mathbb F)$ be a linear variety defined by $\K = \left\{X \in \M_{n\,k}(\F) \mid \mathbf z^t X = \mathbf 0\right\}$. Then 
\begin{equation}\label{thm:RowVarDetNKZeroCond:eqq2}
\det_{n\,k} (X) = 0\;\;\mbox{for all}\;\;X \in \K
\end{equation}
if and only if
\begin{equation}\label{thm:RowVarDetNKZeroCond:eq1}
1 + \sum_{\alpha=1}^{k}\mathbf z_{c(\alpha)} (-1)^{\alpha - c(\alpha)} = 0\quad\mbox{for all}\;\; c \in \binom{[n]}{k}\;\;\mbox{such that}\;\; 1 \not\in c.
\end{equation}
\end{lemma}
\begin{proof}First we prove the necessity and then we prove the sufficiency.

\paragraph{Necessity.} Let $c \in \binom{[n]}{k}$ be such that $1 \not\in c$. Consider $X_{c} \in \M_{n\,k}(\F)$  defined by
\[
X_{c} = \sum_{\alpha=1}^{k} \left(E_{c(\alpha)\,\alpha} + \mathbf z_{c(\alpha)}E_{1\,\alpha}\right).
\]
The matrix $X_{c}$ clearly satisfies the equation $\mathbf z^t X_{c} = 0$ and consequently $X_{c} \in \K$. In addition, for every $c' \in \binom{[n]}{k}$ such that $1 \not\in c'$
\begin{equation}\label{thm:RowVarDetNKZeroCond:eq103}
\det_k(X_{c}[c'|)) =
\begin{cases}
1, & c' = c\\
0, & \mbox{otherwise}
\end{cases}.
\end{equation}
On the one hand, 
\begin{equation}\label{thm:RowVarDetNKZeroCond:eq102}
\det_{n\,k}(X_{c}) = \sum_{\substack{c' \in \binom{[n]}{k}\\1 \not\in c'}} \left(1 + \sum_{\alpha=1}^{k}\mathbf z_{c'(\alpha)} (-1)^{\alpha - c'(\alpha)}\right)\sgn_{[n]}(c')\det_k\Bigl(X_c[c'|)\Bigr)
\end{equation}
by Lemma~\ref{thm:RowVarDetNKZeroExpress}. By substituting~\eqref{thm:RowVarDetNKZeroCond:eq103} into~\eqref{thm:RowVarDetNKZeroCond:eq102} we obtain that
\begin{equation}\label{thm:RowVarDetNKZeroCond:eq104}
\det_{n\,k}(X_{c}) = \left(1 + \sum_{\alpha=1}^{k}\mathbf z_{c(\alpha)} (-1)^{\alpha - c(\alpha)}\right)\sgn_{[n]}(c).
\end{equation}
On the other hand, since $X_{c} \in \K$, then 
\begin{equation}\label{thm:RowVarDetNKZeroCond:eq105}
\det_{n\,k}(X_{c}) = 0.
\end{equation}
Hence, by aligning~\eqref{thm:RowVarDetNKZeroCond:eq104} and~\eqref{thm:RowVarDetNKZeroCond:eq105} together we conclude that
$$\left(1 + \sum_{\alpha=1}^{k}\mathbf z_{c(\alpha)} (-1)^{\alpha - c(\alpha)}\right)\sgn_{[n]}(c) = 0.$$
Dividing this equality by $\sgn_{[n]}(c) \neq 0$ yields the equality
\[
1 + \sum_{\alpha=1}^{k}\mathbf z_{c(\alpha)} (-1)^{\alpha - c(\alpha)} = 0.
\]
Since this equality holds for every $c \in \binom{[n]}{k}$ such that $1 \not\in c$, the \textbf{Necessity} part of the lemma is established.

\paragraph{Sufficiency.} Let $X \in \K$. By Lemma~\ref{thm:RowVarDetNKZeroExpress} we have 
\begin{equation}\label{thm:RowVarDetNKZeroCond:eq101}
\det_{n\,k}(X) = \sum_{\substack{c \in \binom{[n]}{k}\\1 \not\in c}} \left(1 + \sum_{\alpha=1}^{k}\mathbf z_{c(\alpha)} (-1)^{\alpha - c(\alpha)}\right)\sgn_{[n]}(c)\det_k\Bigl(X[c|)\Bigr).
\end{equation}
By substituting~\eqref{thm:RowVarDetNKZeroCond:eq1} into~\eqref{thm:RowVarDetNKZeroCond:eq101} we obtain that
$$
\det_{n\,k}(X) =  \sum_{\substack{c \in \binom{[n]}{k}\\1 \not\in c}} 0\cdot \sgn_{[n]}(c)\det_k\Bigl(X[c|)\Bigr) = 0.
$$
Thus, $\det_{n\,k}(X) = 0$ for all $X \in \K$ which establishes the \textbf{Sufficiency} part of the lemma.
\end{proof}

\begin{theorem}\label{thm:MaxDimKEvenAltSumZero}
Let $n \ge k + 2$ and $\K \subseteq \M_{n\,k}(\mathbb F)$ be a linear variety. Then $\det_{n\,k} (X) = 0$ for all $X \in \K$ and $\codim(\K) = k$ if and only if $k$ is odd and alternating row sum of every $X \in \K$ is equal to zero.
\end{theorem}

\begin{proof}First we prove the necessity and second we prove the sufficiency.
\paragraph{Necessity.} Let $\K$ be a linear variety satisfying the conditions of the lemma. Lemma~\ref{lem:ALSDetNKZeroDimKThenExistVec} implies that there exists $\mathbf z \in \F^n$ such that 
\begin{equation}\label{thm:MaxDimKEvenAltSumZero:eq11}
\mathbf z^t X = \mathbf 0\;\;\mbox{for all}\;\; X \in \K. 
\end{equation}

The proof is built up from the following three claims proven independently for all $n, k$ and $\K$ satisfying the conditions of the lemma:
%Now we prove sequently the following three claims:
\begin{enumerate}[label=(\roman*), ref=(\roman*)]
\item\label{thm:MaxDimKEvenAltSumZero:part3} If $\mathbf z_1 = -1$, then 
\begin{equation}\label{thm:MaxDimKEvenAltSumZero:eq15}
\mathbf z_{i} = -\mathbf z_{i+1}\;\;\mbox{for all}\;\;1 < i < n.
\end{equation}
\item\label{thm:MaxDimKEvenAltSumZero:part4} The equality~\eqref{thm:MaxDimKEvenAltSumZero:eq15} implies the statement of this part the lemma.
\item\label{thm:MaxDimKEvenAltSumZero:part2} If the claims~\ref{thm:MaxDimKEvenAltSumZero:part3} and~\ref{thm:MaxDimKEvenAltSumZero:part4} hold, then this part of the lemma holds (we mean the \textbf{Necessity} part).
\end{enumerate}
Thus, this part of the lemma follows from the claims~\ref{thm:MaxDimKEvenAltSumZero:part3}--\ref{thm:MaxDimKEvenAltSumZero:part2}.

\paragraph{Claim~\ref{thm:MaxDimKEvenAltSumZero:part3}.} 
To establish the equality~\eqref{thm:MaxDimKEvenAltSumZero:eq15}, consider any subsequence $m_1, \ldots, m_{k+1}$ of the sequence $2, \ldots, n$ containing both $i$ and $i+1$ (such a sequence exists because $n \ge k+2$). Suppose that $i = m_{l}.$ Then consider  $c_1 = \{m_1,\ldots,\widehat{m_{l+1}},\ldots,m_{k+1}\} \in \binom{[n]}{k}$ and $c_2 = \{m_1,\ldots,\widehat{m_{l}},\ldots,m_{k+1}\} \in \binom{[n]}{k}$. Since $1 \not\in c_1, c_2$, the equalities
\begin{equation}\label{eqexp1}
1 + \sum_{1 \le \alpha \le k}(-1)^{\alpha - c_1(\alpha)} {\mathbf z}_{c_1(\alpha)} = 0\quad\mbox{and}\quad
1 + \sum_{1 \le \alpha \le k}(-1)^{\alpha - c_2(\alpha)} {\mathbf z}_{c_2(\alpha)} = 0
\end{equation}
hold by Lemma~\ref{thm:RowVarDetNKZeroCond}. Since $c_1(\alpha) = c_2(\alpha)$ for all $\alpha \neq l$, the difference of the left-hand sides of equalities in~\eqref{eqexp1} is equal to
\[
(-1)^{l - c_1(l)} {\mathbf z}_{c_1(l)} - (-1)^{l - c_2(l)} {\mathbf z}_{c_2(l)} = (-1)^l\left((-1)^{i} {\mathbf z}_i - (-1)^{i+1}{\mathbf z}_{i+1}\right).
\]
Hence,
\[
{\mathbf z}_i + {\mathbf z}_{i+1} = 0.
\]
This implies the equality~\eqref{thm:MaxDimKEvenAltSumZero:eq15}.

\paragraph{Claim~\ref{thm:MaxDimKEvenAltSumZero:part4}.} 
From \ref{thm:MaxDimKEvenAltSumZero:part3} we conclude that there exists $z \in \F$ such that 
\begin{equation}\label{thm:MaxDimKEvenAltSumZero:eq201}
\mathbf z_i = (-1)^{i}z\;\;\mbox{for all}\;\;i \ge 2.
\end{equation}
Lemma~\ref{thm:RowVarDetNKZeroCond} applied to  $c = \{2, \ldots, k+1\}$ implies that
$1 + \sum_{1 \le \alpha \le k}(-1)^{\alpha + c(\alpha)} \mathbf z_{c(\alpha)} = 0$
and consequently
\begin{equation}\label{thm:MaxDimKEvenAltSumZero:eq202}
\sum_{1 \le \alpha \le k}(-1)^{\alpha + c(\alpha)} \mathbf z_{c(\alpha)} = -1.
\end{equation}
Consider the left-hand side of~\eqref{thm:MaxDimKEvenAltSumZero:eq202}. By the equality~\eqref{thm:MaxDimKEvenAltSumZero:eq201} we obtain that
\begin{equation}\label{thm:MaxDimKEvenAltSumZero:eq203}
\sum_{1 \le \alpha \le k}(-1)^{\alpha + c(\alpha)} \mathbf z_{c(\alpha)} = \sum_{1 \le \alpha \le k}(-1)^{\alpha - \alpha-1} (-1)^{i+1}z = \sum_{1 \le \alpha \le k}(-1)^{i}z.
\end{equation}
Substituting the right-hand side of~\eqref{thm:MaxDimKEvenAltSumZero:eq203} into~\eqref{thm:MaxDimKEvenAltSumZero:eq202} yields
\begin{equation}\label{thm:MaxDimKEvenAltSumZero:eq4}
\sum_{1 \le \alpha \le k}(-1)^{i}z = -1.
\end{equation}

If $k$ is even, we obtain a contradiction because the left part of~\eqref{thm:MaxDimKEvenAltSumZero:eq4} is zero while the right part is nonzero.

If $k$ is odd, then  the left part of~\eqref{thm:MaxDimKEvenAltSumZero:eq4} is equal to $-z$, which implies that $z = 1$. Thus, $\mathbf z_{i} = (-1)^{i}$.

\paragraph{Claim~\ref{thm:MaxDimKEvenAltSumZero:part2}.} Since $\mathbf z \neq 0$, there exists $1 \le i^{\circ} \le n$ such that $\mathbf z_{i^{\circ}} \neq 0$. The equality~\eqref{thm:MaxDimKEvenAltSumZero:eq11} implies that 
\begin{equation}\label{thm:MaxDimKEvenAltSumZero:eq204}
{{\mathbf z}'}^t X = \mathbf 0\;\;\mbox{for all}\;\;X \in \K,
\end{equation}
 where $\mathbf z'$ is defined by
$
\mathbf z'_{i} = -\frac{{{\mathbf z}}_{i}}{{{\mathbf z}}_{i^{\circ}}}\;\;\mbox{for all}\;\; 1 \le i \le n.
$
Additionally, 
\begin{equation}\label{thm:MaxDimKEvenAltSumZero:eq205}
\mathbf z'_{i^{\circ}} = -1
\end{equation}
 by the definition of $\mathbf z'$.

%
%\begin{underconst}
%Let $\SCS_i$ be defined as it is done in Lemma~XXXXX
 
%\end{underconst}
By Lemma~\ref{lem:SCS}, there exists an invertible linear map $\SCS_{i^{\circ}} \colon \M_{n\, k}(\F) \to \M_{n\,k}(\F)$ having the following properties:
\begin{enumerate}[label=(S\arabic*), ref=(S\arabic*)]
%\item%\label{SCSprop1}
% $\left(\SCS_{i^{\circ}}(X)\right)[1|) = X[i|)$ for all $X \in \M_{n\,k}(\F)$;
\item%\label{SCSprop2}
if $X \in \M_{n\,k}(\F)$ and $\mathbf z \in \F^n$ are such that $\mathbf z^t X = 0$, then $\mathbf {z^{\circ}}^t\SCS_{i^{\circ}}(X) = 0$, where 
    \begin{equation*}\mathbf z^{\circ} = (\mathbf z_{i^{\circ}}, \ldots, \mathbf z_n, (-1)^{n+k+1}\mathbf z_1 \ldots, \mathbf (-1)^{n+k+1}\mathbf z_{i^{\circ}-1})^t;
    \end{equation*}
\item%\label{SCSprop3}
$\det_{n\,k}(X) = 0 \Rightarrow \det_{n\,k}(\SCS_{i^{\circ}}(X)) = 0$;
\item%\label{SCSprop4}
If $k$ is odd, then 
$$\sum_{i=1}^n (-1)^{i}\left(\SCS_{i^{\circ}}(X)\right)[i|) = \begin{pmatrix}0 & \cdots & 0\end{pmatrix} =\joinrel=\joinrel\Rightarrow \sum_{i=1}^n (-1)^{i}X[i|) = \begin{pmatrix}0 & \cdots & 0\end{pmatrix}$$
for all $X \in \M_{n\,k}(\F).$
\end{enumerate}
Let $K^{\circ}$ be defined by $\K^{\circ} = \SCS_{i^{\circ}}(\K)$. Since $\SCS_{i^{\circ}}$ is an invertible linear map, $\K^{\circ}$ is a linear variety with $\codim(\K^{\circ}) = \codim(\K) = k$. By the property~\ref{SCSprop3} of $\SCS_{i^\circ}$ we have $\det_{n\,k}(X^{\circ}) = 0$ for all $X^{\circ} \in \K^{\circ}.$

It follows from the property~\ref{SCSprop2} of $\SCS_{i^{\circ}}$ and~\eqref{thm:MaxDimKEvenAltSumZero:eq204} that $\mathbf z^{\circ} X^{\circ} = 0$ for all $X^{\circ} \in \K^{\circ}$, where $\mathbf{z}^{\circ}$ is defined by
\begin{equation}\label{thm:MaxDimKEvenAltSumZero:eq206}
\mathbf{z}^{\circ} = (\mathbf z'_{i^{\circ}}, \ldots, \mathbf z'_n, (-1)^{n+k+1}\mathbf z'_1 \ldots, \mathbf (-1)^{n+k+1}\mathbf z'_{i^{\circ}-1})^t.
\end{equation}
The equalities~\eqref{thm:MaxDimKEvenAltSumZero:eq205} and~\eqref{thm:MaxDimKEvenAltSumZero:eq206} imply that $\mathbf{z}^{\circ}_1 = -1.$ Hence, we can apply the claims~\ref{thm:MaxDimKEvenAltSumZero:part3} and~\ref{thm:MaxDimKEvenAltSumZero:part4} to $\K^{\circ}$ and conclude that $k$ is odd and
\begin{equation}\label{thm:MaxDimKEvenAltSumZero:eq101}
\sum_{i=1}^n (-1)^{i}X^{\circ}[i|) = 0\;\;\mbox{for all}\;\; X^{\circ} \in \K^{\circ}.
\end{equation}

Let $X \in \K$. Then $\SCS_{i^{\circ}}(X) \in \K^{\circ}$ and consequently
$
\sum_{i=1}^n (-1)^{i}\left(\SCS_{i^{\circ}}(X)\right)[i|) = 0
$ 
by~\eqref{thm:MaxDimKEvenAltSumZero:eq101}. Hence, $\sum_{i=1}^n (-1)^{i}X[i|) = 0$ by the property~\ref{SCSprop4} of $\SCS_{i^\circ}$. That is, the alternating row sum of every element in $\K$ is equal to zero. This establishes the claim~\ref{thm:MaxDimKEvenAltSumZero:part2}.

\paragraph{Sufficiency.} Assume that $k$ is odd and $\K \subseteq \M_{n\,k}(\F)$ is a vector space consisting of all the matrices such that the alternating sum of their rows is zero. It is clear that $\codim(\K) = k$.

Assume that $X \in \K$. Then $\mathbf z^t X = \mathbf 0$ for $\mathbf z = (-1, 1, \ldots, (-1)^n)^t$ by the definition of $\K$. Lemma~\ref{thm:RowVarDetNKZeroCond} implies that in order to show that $\det_{n\,k}(X)=0$ it is sufficient to check the condition~\eqref{thm:RowVarDetNKZeroCond:eq1} for every $c \in \binom{[n]}{k}$ such that $1 \not\in c$. For every such $c$ we have
\begin{equation*}
1 + \sum_{\alpha=1}^{k}\mathbf z_{c(\alpha)} (-1)^{\alpha - c(\alpha)} = 1 + \sum_{\alpha=1}^{k}(-1)^{c(\alpha)} (-1)^{\alpha - c(\alpha)}
= 1 + \sum_{\alpha=1}^{k}(-1)^{\alpha} = 1 - 1 = 0.
\end{equation*}
Therefore, the condition~\eqref{thm:RowVarDetNKZeroCond:eq1} is satisfied  for every $c \in \binom{[n]}{k}$ such that $1 \not\in c$. This implies that $\det_{n\,k} (X) = 0$ for all $X \in \K.$
\end{proof}

\section{Further work}\label{sec:furtherwork}
First, a natural continuation of this work is to complete the investigation by determining the vector spaces of $n\times k$ matrices of maximal dimension that annihilate $\det_{n,k}$ when $k$ is even. As it is mentioned in the introduction, the only available information is that the codimension of such spaces is strictly larger than $k$. No description of such spaces of maximal dimension is currently known for even $n$, nor is there a strict lower bound for the codimension $k$ of such spaces. %There is not only no hypothesis about the description of such spaces for even $k$ of maximal dimension, but also. 
Any progress in this direction will be interesting.

Second, the found complete solution to the problem of maximal spaces annihilating the Cullis' determinant  could  help to solve linear preserver problem for the Cullis' determinant in the full generality, without any restriction on the ground field, in contrast to how it is done in~\cite{Guterman2025} and~\cite{Guterman2025b}. 

%Third, we suggest that the approach developed in this paper may be useful in solving similar problems, for example, in finding the description of vector spaces of matrices of maximal dimension annihilating the matrix permanent. It seems worthwhile to explore this possibility further.  

%\paragraph{Acknowledgements:} The authors want to thank the anonymous referees very much for taking time and efforts
%in improving our manuscript.
\paragraph{Funding information:} The research of the second author was supported by the scholarship of the Center for Absorption in Science, the Ministry for Absorption of Aliyah, the State of Israel.
\paragraph{Author contributions:} All authors have accepted responsibility for the entire content of this manuscript and
consented to its submission to the journal, reviewed all the results, and approved the final version of the
manuscript.
\paragraph{Conflict of interest:} The authors declare no conflicts of interest.
\paragraph{Data availability statement:} Not applicable.

\bibliographystyle{plain}
\bibliography{cullisgeneral}

\end{document}